\documentclass[letterpaper,11pt,reqno]{amsart} 
\usepackage[portrait,margin=1in]{geometry}
\usepackage{systeme}
\usepackage{mathrsfs,xfrac} 
\usepackage[colorlinks=true,linkcolor=blue,citecolor=blue,urlcolor=blue]{hyperref} 
\usepackage{amsmath,amssymb,amsthm,amsfonts,amsbsy,latexsym,dsfont,color} 
\usepackage[numeric,initials,nobysame]{amsrefs} 
\usepackage[foot]{amsaddr}

\usepackage[utf8]{inputenc}
\usepackage[english]{babel}
\usepackage{comment}

\usepackage[textsize=tiny]{todonotes}
\usepackage{regexpatch}
\makeatletter
\xpatchcmd{\@todo}{\setkeys{todonotes}{#1}}{\setkeys{todonotes}{inline,#1}}{}{}
\makeatother

\usepackage{enumerate}
\newenvironment{enumeratei}{\begin{enumerate}[\upshape i)]}{\end{enumerate}}
\newenvironment{enumeratea}{\begin{enumerate}[\upshape a)]\setlength{\itemsep}{1ex}}{\end{enumerate}}

\newenvironment{enumerateA}{\begin{enumerate}[\upshape (A)]\setlength{\itemsep}{1ex}}{\end{enumerate}}

\usepackage{paralist}

\newtheorem{thm}{Theorem}[section]
\newtheorem{lem}[thm]{Lemma}
\newtheorem{cor}[thm]{Corollary}
\newtheorem{prop}[thm]{Proposition}
\newtheorem{defn}[thm]{Definition}

\newtheorem{ass}[thm]{Assumption}

\theoremstyle{definition}
\newtheorem{rem}[thm]{Remark}

\renewcommand{\le}{\leqslant} 
\renewcommand{\ge}{\geqslant} 
\renewcommand{\leq}{\leqslant} 
\renewcommand{\geq}{\geqslant}

\newcommand{\ind}{\mathds{1}}
\newcommand{\eps}{\varepsilon}

\newcommand{\abs}[1]{\left\vert#1\right\vert}
\newcommand{\set}[1]{\left\{#1\right\}}

\newcommand{\ie}{\emph{i.e.,}}

\newcommand{\eg}{\emph{e.g.,}}

\newcommand{\probc}{\stackrel{\mathrm{P}}{\longrightarrow}}

\def\qed{ \hfill $\blacksquare$}  
\let\ga=\alpha \let\gb=\beta \let\gc=\gamma  
     \let\gl=\lambda

\newcommand{\cB}{\mathcal{B}}
\newcommand{\cF}{\mathcal{F}}
\newcommand{\cH}{\mathcal{H}}
\newcommand{\cL}{\mathcal{L}}
\newcommand{\cM}{\mathcal{M}}\newcommand{\cN}{\mathcal{N}}
\newcommand{\cP}{\mathcal{P}}
\newcommand{\cT}{\mathcal{T}}

\newcommand{\cY}{\mathcal{Y}}  

\newcommand{\vA}{\mathbf{A}}

\newcommand{\vQ}{\mathbf{Q}}

\newcommand{\vW}{\mathbf{W}}

\newcommand{\vs}{\mathbf{s}}\newcommand{\vt}{\mathbf{t}}
\newcommand{\vw}{\mathbf{w}}
 
\newcommand{\mvvarpi}{\boldsymbol{\varpi}}
   
\newcommand{\fN}{\mathfrak{N}}
\newcommand{\fP}{\mathfrak{P}}\newcommand{\fR}{\mathfrak{R}}
\newcommand{\fT}{\mathfrak{T}}
\newcommand{\fX}{\mathfrak{X}}

\newcommand{\fm}{\mathfrak{m}}
\newcommand{\fp}{\mathfrak{p}}

\newcommand{\bN}{\mathbb{N}}
\newcommand{\bR}{\mathbb{R}}
\newcommand{\bT}{\mathbb{T}}

\newcommand{\bZ}{\mathbb{Z}}        

\DeclareMathOperator{\E}{\mathds{E}}
\DeclareMathOperator{\pr}{\mathds{P}}

\DeclareMathOperator{\var}{Var}

\DeclareMathOperator{\N}{N}

\newcommand{\bbT}{\mathbb{T}}
\newcommand{\TT}{\mathcal{T}}
\newcommand{\bs}{\mathbf{s}}

\newcommand{\bt}{\mathbf{t}}
\newcommand{\bfomega}{{\boldsymbol \omega}}
\newcommand{\Zbold}{{\mathbb{Z}}}
\newcommand{\prob}{\mathbb{P}}
\definecolor{aqua}{rgb}{0.0, 1.0, 1.0}
\definecolor{boo}{rgb}{1.0, 0.0, 1.0}

\newcommand{\probfr}{\stackrel{\mbox{$\operatorname{P}$-\bf fr}}{\longrightarrow}}
\newcommand{\probcrf}{\stackrel{\mbox{$\operatorname{P}$-\bf efr}}{\longrightarrow}}

\newcommand{\Efr}{\stackrel{\mbox{$\E$-\bf fr}}{\longrightarrow}}

\newcommand{\convas}{\stackrel{\mathrm{a.s.}}{\longrightarrow}}
\newcommand{\convd}{\stackrel{d}{\longrightarrow}}
\newcommand{\convp}{\stackrel{P}{\longrightarrow}}

\newtheorem{lemma}[thm]{Lemma}

\newcommand{\stod}{\preceq_{\mathrm{st}}}
 \DeclareMathOperator{\BP}{BP}

\newcommand{\spec}{{\sf spectral} }
\newcommand{\fringe}{{\sf fringe} }

\usepackage[mathscr]{euscript}
\DeclareMathAlphabet{\mathscrbf}{OMS}{mdugm}{b}{n}

\newcommand{\sch}{{\mathcalb{h}}}

\newcommand{\tikzcircle}[2][blue,fill=blue]{\tikz[baseline=-0.5ex]\draw[#1,radius=#2] (0,0) circle ;}%
\newcommand{\cic}{\tikzcircle{1.5pt}}

\DeclareFontFamily{U}{BOONDOX-calo}{\skewchar\font=45 }
\DeclareFontShape{U}{BOONDOX-calo}{m}{n}{
  <-> s*[1.05] BOONDOX-r-calo}{}
\DeclareFontShape{U}{BOONDOX-calo}{b}{n}{
  <-> s*[1.05] BOONDOX-b-calo}{}
\DeclareMathAlphabet{\mathcalb}{U}{BOONDOX-calo}{m}{n}
\SetMathAlphabet{\mathcalb}{bold}{U}{BOONDOX-calo}{b}{n}
\DeclareMathAlphabet{\mathbcalb}{U}{BOONDOX-calo}{b}{n}

\newcommand{\sss}{\scriptscriptstyle}
\newcommand{\fpm}{\mvvarpi}
\newcommand{\fplus}[1]{\lfloor #1\rfloor}
\usepackage{tikz}
\usetikzlibrary{calc,intersections,through,backgrounds,shapes.geometric}
\usetikzlibrary{graphs}

\usetikzlibrary{calc,intersections,through,backgrounds,shapes.geometric}
\usetikzlibrary{graphs}
\tikzset{every path/.style={line width=.07 cm}}

\usepackage{subcaption}

\newcommand{\cg}{}

\begin{document}
\title[Network evolution with mesoscopic delay]{Network evolution with Mesoscopic Delays}
\author[Banerjee]{Sayan Banerjee$^1$}
\author[Bhamidi]{Shankar Bhamidi$^1$}
\address{$^1$Department of Statistics and Operations Research, 304 Hanes Hall, University of North Carolina, Chapel Hill, NC 27599}
\email{sayan@email.unc.edu, bhamidi@email.unc.edu, psdey@illinois.edu, sakshay@unc.edu}
\author[Dey]{Partha Dey$^2$}
\address{$^2$Department of Mathematics, University of Illinois Urbana-Champaign, IL 61801}
\author[Sakanaveeti]{Akshay Sakanaveeti$^1$}
\begin{abstract}
	Owing to the influence of real-world networks both in science and society, numerous mathematical models have been developed to understand the structure and evolution of these systems, particularly in a temporal context. Recent advancements in fields like distributed cyber-security and social networks have spurred the creation of probabilistic models of evolution, where individuals make decisions based on only partial information about the network's current state. This paper seeks to explore models incorporating \emph{network delay}, where new participants receive information from a time-lagged snapshot of the system. In the context of mesoscopic network delays, we develop probabilistic tools built on stochastic approximation to understand asymptotics of both local functionals, such as local neighborhoods and degree distributions, as well as global properties, such as the evolution of the degree of the network's initial founder. A companion paper~\cite{BBDS04_macro} explores the regime of macroscopic delays in the evolution of the network.  
\end{abstract}
\maketitle



\section{Introduction}
\label{sec:intro}
Driven by the explosion in the amount of data on various real-world networks, the last few years have seen the development of a number of new mathematical network models to understand the emergence of macroscopic properties in such systems~\cites{albert2002statistical, newman2003structure,newman2010networks, bollobas2001random,durrett-rg-book,van2009random}. One core sub-area in this rapidly burgeoning field is dynamic or temporal networks,~\cites{holme2012temporal, masuda2016guidance}. Motivating questions from, for example, the area of social networks include the impact of these dynamics on information diffusion across social networks, such as the spread of malicious information~\cite{shah2012rumor}, the role of heterogeneity in edge creation across different attribute groups~\cites{wagner2017sampling,espin2018towards,antunes2023learning}, and co-evolution of the dynamics of the network with algorithms such as recommendation systems. Models of network evolution have been crucial in addressing these questions, offering insights into the mechanistic reasons behind empirically observed properties, such as heavy-tailed degree distributions~\cites{albert2002statistical,barabasi1999emergence}, the effectiveness of seed reconstruction algorithms~\cites{bubeck-mossel,bubeck2017finding,banerjee2020persistence}, and the biases inherent in network sampling and ranking algorithms~\cites{wagner2017sampling,karimi2022minorities,karimi2018homophily,antunes2023learning}.

{\cg The preferential attachment model, one major class of network evolution schemes to understand the various phenomena described above,  assumes new vertices enter the system and then make probabilistic choices for connecting to existing vertices based on the {\bf entire} current state of the network; in this regime there has been tremendous progress both in the applied domains as well as in probabilistic combinatorics, see e.g.~\cites{durrett-rg-book,van2009random,van2023random,dorogovtsev2002evolution,barabasi1999emergence,krapivsky2001organization,krapivsky2005network,newman2003structure,newman2010networks}. An area that has witnessed significant interest in the applied community but remains largely unexplored within probabilistic combinatorics is the impact of limited information availability on the system's state for new individuals attempting to form connections. This vacuum has motivated (at least) two distinct lines of recent research:
\begin{enumeratea}
\item  Network evolution schemes where individuals try to form connections via limited scope explorations such as random walks,  initialized at uniform locations in the network~\cites{krapvisky2023magic, banerjee2022co}. 

\item  Most relevant for this work, incorporating network delay as a mechanism for information limitation where new individuals have information on a delay modulated snapshot of the network state~\cites{baccelli2019renewal,king2021fluid,dey2022asymptotic} and have motivated a number of random models of directed acyclic graphs \cite{computers12120257,xiao2024accelerating,ahmed2024optimization} including relatives of the preferential attachment class of networks \cite{monch2021dag}.   This paper and the companion paper~\cite{BBDS04_macro} answer question 6~in~\cite{dey2022asymptotic}.

\end{enumeratea}

We start by describing the model in this paper and then outline the main contributions of this paper in Section \ref{sec:outline-contrib}. 
}

\subsection{Network model}
\label{sec:net-mod}
Let us now describe the precise model, initially restricting to the tree network setting and commenting on the more general setting subsequently. The model has three ingredients:
\begin{enumeratea}
\item A \emph{time scale parameter} $\gb \in [0,1]$. This paper deals with the {\bf mesoscopic} regime, where $\gb \in (0,1)$, named for the reasons that will become obvious after the model definition. The regime $\gb= 0$ will be referred to as the {\bf microscopic} delay.   The regime $\gb = 1$, referred to as the {\bf macroscopic} delay regime, is considered in a companion paper~\cite{BBDS04_macro}.   
\item A \emph{delay distribution} $\mu$ on {\cg $(\bR_+, \cB(\bR_+))$}. 
\item An \emph{attachment function} $f:\set{0,1,\ldots} \to \bR_+$, measuring the attractiveness of individuals based on their local information,  assumed to be strictly positive.
\end{enumeratea}

\begin{defn}[Network evolution with delay]
\label{defn:model}We generate a sequence of growing random trees $\set{\cT(n):n\geq 1 }$ as follows:  

\begin{enumeratei}
	\item The initial tree $\cT(1)$ consists of a single vertex, $v_1$. At time $2$, the tree $\cT(2)$ consists of two vertices, labeled as $\set{v_1, v_2}$ and attached by a single edge directed from $v_2$ to $v_1$. Call the vertex $v_1$ the ``root'' of the tree. 
 
	\item Suppose the network has been constructed till time $n$ for $n\geq  2$. Let the vertices in $\cT({n})$ be labeled as $\{v_1,\dots,v_{n}\}$ with edges directed from children to parent. For vertex $v_i\in \cT({n})$ (which entered the system at time $i$)  and $j\geq  i$, let $\deg(v_i,j)$ denote the degree of $v_i$ at time $j$ (which, for all vertices other than the root, is equal to the in-degree $+ 1$) interpreted in various applications as a measure of trust accumulated till time $i$. Initialize always with $\deg(v_i,i) =1$.

	\item  At time  $(n+1)$, a new vertex $v_{n+1}$ enters the system. A \emph{normalized} time delay, $\xi_{n+1} \sim \mu$, independent of $\cT(n)$, is sampled. The information available to $v_{n+1}$ is the graph $\cT(\fplus{n - n^{\gb} \xi_{n+1}}_+)$. 
	\item Conditional on $\cT(\fplus{n - n^{\gb}\xi_{n+1}}_+)$, this new vertex attaches to a vertex $v\in \cT(\fplus{n - n^{\gb}\xi_{n+1}}_+)$ with probability proportional to 
	\[
        f(\deg(v,\fplus{n - n^{\gb}\xi_{n+1}}_+)).
    \]
\end{enumeratei}

Write $\set{\cT(n):n\geq 1}$ for the corresponding sequence of growing random trees and let $\cL(\gb, \mu, f)$ denote the corresponding probability distribution of the sequence of growing random trees.  
\end{defn}

For notational convenience, we will write $\fplus{x}=\max\{\lfloor x\rfloor,1\}$, $\cT(0) = \cT(1)$ and set  $\deg(v_i,j) = 0$ for $j<i$. 
In the evolution dynamics above,  vertex $v_{n+1}$ only has the information of the state of the process  $\cT(\fplus{n - n^{\gb}\xi_{n+1}}_+)$. This includes the vertex set only at that time and their corresponding degree information (not the full degree information at time $n$). Call the vertex that $v_{n+1}$ attaches to, the ``parent'' of $v_{n+1}$ and direct the edge from $v_{n+1}$ to this parent resulting in the tree $\cT(n+1)$. 
Thus, for any $n$, $\cT(n)$ is a random tree on $n$ labelled vertices $\{v_1,\dots,v_n\}$ rooted at $v_1$. In the sequel, to reduce notational overhead, we will drop the ``integer part'' $\lfloor \cdot \rfloor$. We will use $\xi$ to denote a typical r.v.~having distribution $\mu$.
To keep this paper to a manageable length, we focus on the tree-network setting,  deferring further discussion of potential extensions to the non-tree setting,  related work, as well as open problems to Section~\ref{sec:conc}. 

Next, note that $\gb \equiv 1$ corresponds to the setting where the delay $n\xi_{n+1}$ is of the same order as the size of the network and thus is referred to as the \emph{macroscopic} delay regime and is studied in~\cite{BBDS04_macro}. For the rest of the paper, we will assume $\gb < 1$. This regime, and in fact the base case with $\beta \equiv 0$, was the original setting considered in understanding the evolution of systems such as blockchains using probabilistic models for directed acyclic graphs, where one can observe a nontrivial effect on the structure of the graph due to the delay. For example, the blockchain graph structure changes from one-ended to many-ended under the Nakamoto protocol with nondegenerate microscopic ($\beta\equiv0$) delay. Moreover, all the proofs for the microscopic delay regime easily extend to the mesoscopic regime. Hence, in this article, we consider the more general regime $\beta \in [0,1)$. See \cites{nakamoto2008bitcoin,popov2018tangle} for the original formulations in areas such as blockchains and Internet of Things (IOT); see \cites{king2021fluid,muller2022tangle,li2019markov} for recent probabilistic formulations of the mechanisms described in the original papers, and see \cites{king2021fluid,dey2022asymptotic,fk23} specifically dealing with the microscopic regime.  

{\cg Observe that we have taken a specific form of delay which grows as $n^\beta$ in the network size $n$. One could consider a delay of the form $n - \beta$, which would imply every incoming vertex having access to the initial network of size $\lfloor \beta \rfloor$. This would result in a trivial kind of \emph{condensation}: every vertex in this initial network will have a degree proportional to the network size. It is easy to check (and can also be derived from the results of this paper) that a delay of the form $\beta$ (not growing with $n$) will result in most interesting asymptotics exactly mimicking that of the corresponding network model without delay. Delays of the form $\beta n$ fit into the class of macroscopic delays treated in \cite{BBDS04_macro}. Thus, the only other natural delay, of the form $n^\beta$, is studied in the current paper.}

Explicit examples of attachment functions {\bf without delay} that have been considered in the literature include:

\begin{enumeratea}
	\item {\bf Uniform attachment:} $f(\cdot)\equiv 1$. In this case, new vertices attach to existing vertices uniformly at random. The corresponding model is called the random recursive tree and has been heavily analyzed across probabilistic combinatorics and computer science~\cites{smythe1995survey,drmota2009random}. 
	\item {\bf ``Pure'' preferential attachment:} $f(k) = k$. Here, new vertices attach to pre-existing vertices with probability proportional to the degree of the existing vertex~\cite{barabasi1999emergence}. 
	\item {\bf Affine Preferential attachment:} $f(k) = k+\alpha$ for $\alpha \ge 0$~\cite{durrett2007random}. This is a more general analog of the pure preferential attachment model above, in which one can `fine-tune' the attachment probability of the lower degree vertices via the parameter $\alpha$.
\end{enumeratea}

\subsection{Outline of our contributions}
\label{sec:outline-contrib}
{\cg 
There are two major contributions of this paper: 
\begin{enumeratea}
    \item {\bf Development of stochastic approximation tools for local weak convergence:} The first major goal of this paper is to formulate the delay model above and understand large network asymptotics for the model; a formal background for and statements of the main results are developed in the following two sections. In brief, the goal of this paper is to understand not just local functionals such as the degree but, in fact, the entire local neighborhood of typical vertices in preferential attachment models with general functions of the degree driving the evolution, coupled with delay. One major contribution of this paper is the development of stochastic approximation techniques to prove the local weak convergence of such models towards limiting infinite objects via leveraging the historical ordering of the growth of such trees. This method not only gives asymptotics of local functionals but also leads to asymptotics for global functionals, such as the spectral distribution of the adjacency matrix. 
    \item {\bf Delineation of functionals that are robust or vulnerable to mesoscopic delays:} The mathematical techniques developed in this paper reveal that, in this regime, 
    \begin{enumeratei}
    \item Under regularity conditions on the attachment function $f$ and mild moment conditions on the delay distribution $\mu$,  {\bf irrespective} of the parameter $\gb < 1$ and delay distribution $\mu$, we show that the local weak limit of the entire graph {\bf is the same} as in the setting without delay.
    \item  We study the asymptotics of other functionals, such as second-order fluctuations of the degree count, leading to intriguing conjectures for potential phase transitions as the level of mesoscopic delay transitions from $\beta < 1/2$ to $\beta > 1/2$. 
    \item We analyze the evolution of the root degree (as a proxy for more complex `global' functionals, such as the maximal degree),  deriving conditions for the scaling of the root degree to feel the effect of the delay.  
\end{enumeratei}
\end{enumeratea}

}

\subsection{Organization of the paper}
 Section~\ref{sec:main-res} contains statements of the main results and a brief overview of related work. For readers unfamiliar with local weak convergence and fringe convergence, precise definitions are given in Section~\ref{sec:local-wll}. Section~\ref{sec:proofs-meso} contains proofs of the main results. Finally, we conclude with open problems and potential extensions in Section~\ref{sec:conc}.

\section{Results}
\label{sec:main-res}
We first set up the basic probabilistic objects required to state our main results. 

\subsection{Continuous time branching processes and mesoscale limits}
\label{sec:ctbp-def}
Here, we set up the ingredients required to describe local weak limits in the mesoscopic regime. We will need the following assumption on the attachment functions of interest in this paper. We mainly follow~\cites{jagers-ctbp-book,jagers1984growth,nerman1981convergence,rudas2007random}. 

\begin{ass}
	\label{ass:attach-func} 
 Attachment function $f$ is assumed to satisfy:
	\begin{enumeratei}
		\item $f_* := \inf_{i\geq 1} f(i) > 0$. 
		\item $f$ can grow at most linearly i.e., $\exists\ C< \infty$ such that $\limsup_{k\to\infty} f(k)/k \leq C$ (equivalently there exists a constant $C^\prime < \infty$ such that $f(k)\leq C^\prime k$ for all $k\geq 1$). 
		\item Consider the following function $\hat{\rho}:(0,\infty)\to (0,\infty]$ defined via,
\begin{align}
\label{eqn:rho-hat-def}
	\hat{\rho}(\gl):= \sum_{k=1}^\infty \prod_{i=1}^{k} \frac{f(i)}{\gl + f(i)}. 
\end{align}
 Define 
$\underline{\gl}:= \inf\set{\gl > 0: \hat{\rho}(\gl) < \infty}$. 
We assume that
\begin{align}
\label{eqn:prop-under-lamb}
\underline{\gl} < \infty \ \text{ and } \ \lim_{\gl\downarrow\underline{\gl}} \hat{\rho}(\gl) > 1. 
\end{align}
\end{enumeratei}
\end{ass}
Using iii) in the above assumption and the monotonicity of $\hat{\rho}(\cdot)$, there exists a unique $\gl^*:=\gl^*(f) \in (0, \infty)$ such that 
\begin{align}
\label{eqn:malthus-def}
	\hat{\rho}(\gl^*) = 1. 
\end{align}
This object is often referred to as the Malthusian rate of growth parameter~\cites{jagers-nerman-1,jagers-nerman-2,jagers-ctbp-book}. The above assumptions are relatively standard and are required for analyzing an associated branching process, which we now describe. We defer further assumptions on the attachment function when describing the main results. Fix an attachment function $f$. We now use $f$ to construct a point process $\cP_f$ on $\bR_+$. Generate a sequence of independent exponential random variables $\set{E_i:i\geq 1} $ with the rate of $E_i \sim \exp(f(i))$. Next, define 
$\sigma_i:= \sum_{j=1}^{i} E_i$ for $ i\geq 1 $ with $\sigma_0 =0$.
The point process $\cP_f$ is defined via, 

\begin{align}
\label{eqn:xi-f-def}
	\cP_f:=(\sigma_1, \sigma_2, \ldots).
\end{align}
Abusing notation, write for $u\geq 0$,
\begin{align}
\label{eqn:xi-f-t}
	\cP_f[0,u]:= \abs{\set{i: \sigma_i \leq u}}, \text{ and } \mu_f[0,u]:= \E(\cP_f[0,u]). 
\end{align}
Notice that, $\cP_f, \mu_f$ can be naturally extended to measures on $(\bR_+, \cB(\bR_+))$.

\begin{defn}[Continuous time branching process (CTBP) {\cites{jagers-ctbp-book,athreya1972,nerman1981convergence}}]
	\label{def:ctbp}
	Fix an attachment function $f$ satisfying Assumption~\ref{ass:attach-func}(ii). A continuous time branching process driven by $f$, written as $\{\BP_f(u): u \ge 0\}$, is defined to be a branching process started with one individual at time $u=0$ and such that this individual, as well as every individual born into the system, has an offspring distribution that is an independent copy of the point process $\cP_f$ defined in~\eqref{eqn:xi-f-def}.  
\end{defn} 
Next, we present a fundamental distribution on the space of finite rooted trees, which characterizes the asymptotic behavior of local properties in such branching processes. For this, we need some notation. For $n\geq 1$, let $ \bT_{n} $ be the space of all rooted trees on  $n$ {vertices}. Let $ \bbT =
\cup_{n=0}^\infty \bT_{n} $ be the space of all finite rooted trees.  Here $\bT_{0} = \emptyset $ will be used to represent the empty tree (tree on zero vertices). For any $\bt \in \bbT$, let $\rho_{\bt}$ denote the root this tree.

\begin{defn}[Stable age distribution, {\cites{jagers-ctbp-book,jagers-nerman-1,jagers-nerman-2,nerman1981convergence}}]
\label{def:limit-bp-meso}
	Let $\BP_f(\cdot)$ be a continuous-time branching process as above and let $\gl^*$ be the associated Malthusian rate of growth as in~\eqref{eqn:malthus-def}. Let $T_{\gl^*}$ be an $\exp(\gl^*)$ random variable, independent of $\BP_f$ and write $\fpm_{\BP_f}$ for the distribution of $\BP_f(T_{\gl^*})$, viewed as a random finite rooted tree on $\bbT$, where we retain only genealogical information between individuals in $\BP_f(T_{\gl^*})$. This is called the stable age distribution of the continuous-time branching process $\BP_f(\cdot)$. 
\end{defn}
For later use in the proof, $\fpm_{\BP_f}$ has several special properties including arising as the unique probability distribution on the space $\bT$ satisfying a specific recursive construction (see Prop. \ref{prop:prop-of-fp} below and \cite{rudas2007random}), and as shown {\cg in~\cite[Sections 2.4 and 4.5]{aldous-fringe}},  is an example of an \emph{extremal} fringe distribution as in Definition~\ref{fringedef} and the discussion following it.

\subsection{Statement of the main results}

We will now describe our results; the same proofs work in the microscopic ($\gb = 0$, so $O_P(1)$ delays), so we do not explicitly state results in the microscopic regime. We start with the final set of assumptions we need on the attachment functions and delay distribution. 

\begin{ass}[Delay and attachment function]
\label{ass:lipschitz}
In addition to Assumptions~\ref{ass:attach-func} on the attachment function, we assume the following.
\begin{enumeratea}
\item {\bf Assumptions on the delay distribution:}
Assume that 
\begin{equation}
        \label{eqn:815}
        \lim_{n\to \infty} \E\left[\frac{n^{\gb}\xi\cdot \ind{\{ n -n^{\gb}\xi \geq 1\}}}{\fplus{n-n^{\gb}\xi}}\right] =0.
    \end{equation}
    \item {\bf Assumptions on the attachment function: } Assume that one of the three conditions holds:\medskip

\begin{enumerateA}
\item {\bf Linear:} $f(k) = k+\alpha$ for all $k\geq 1$.
    \item {\bf Sublinear and Lipschitz continuity:} $f$ is sublinear i.e., $\lim_{k\to \infty} f(k)/k = 0$ and  $f$ is Lipschitz with some constant $L< \infty$.
    \item {\bf Sublinear and non-decreasing:} $f$ is sublinear and non-decreasing.  
\end{enumerateA}
\end{enumeratea}
\end{ass}

The following Lemma gives an easy sufficient condition for the delay distribution to satisfy the assumptions above, as well as another condition required in the analysis of the root degree later in Theorem~\ref{thm:meso-max-degree}.

\begin{lem}\label{lem:815} 
Fix $\gb\in [0,1)$.
   \begin{enumeratea}
       \item Assume that the delay distribution satisfies $\lim_{n\to\infty}n\log n \cdot \pr(\lceil\xi^{\frac{1}{1-\gb}} \rceil =n)=0$. Then~\eqref{eqn:815} holds. 
        \item Assume that $\E[\log_+\xi]<\infty$. Then, we have
        \begin{equation}
    \label{eqn:1257}
    \sum_{n=1}^\infty \frac{1}{n}\E\left[\frac{n^{\gb}\xi \cdot \ \ind{\{ n -n^{\gb}\xi \geq 1\}}}{\fplus{n-n^{\gb}\xi}}\right] < \infty. 
    \end{equation}
   \end{enumeratea}
\end{lem}

\begin{rem}
    Note that the condition in Lemma~\ref{lem:815}-a) is almost optimal, as we do need $\lim_{n\to\infty}n\pr(\lceil \xi^\frac{1}{1-\gb}\rceil=n)=0$ for equation~\eqref{eqn:815} to hold.
\end{rem}

Now, let us state the main result of this section. 

\begin{thm}[Local weak limit, mesoscopic regime]\label{thm:meso-local}
	Consider the sequence of random trees $\set{\cT(n):n\geq 1} \sim \cL(\gb, \mu, f)$, where $\gb <1$, and both the attachment function $f$ and delay distribution $\mu$ satisfy Assumptions~\ref{ass:attach-func} and~\ref{ass:lipschitz}. Then $\set{\cT(n):n\geq 1}$ converges in probability in the extended fringe sense (Def.~\ref{def:local-weak}~\eqref{it:fringe-b}) to the unique infinite {\tt sin}-tree with fringe distribution $\fpm_{\BP_f}$. 
\end{thm}

We describe two implications of the above result.

\begin{cor}
\label{cor:meso-deg}
Under the assumptions of Theorem~\ref{thm:meso-local}, for the random tree models $\set{\cT(n):n\geq 2} \sim \cL(\gb, \mu, f)$ the degree counts $\set{N_k(n):k\geq 1}$ as in~\eqref{eqn:deg-count}  satisfy
\[
\frac{N_k(n)}{n} \convas p_k(f) \text{ as } n\to\infty,
\]
where 
\[
p_k(f) := \int_0^{\infty} \gl^* e^{-\gl^* t} \mathbb{P}\left(\cP_{f}(t)=k-1\right)dt = \frac{\gl^*}{\gl^* + f(k)} \prod_{j=1}^{k}\frac{f(j)}{\gl^* + f(j)},\quad k\ge 1.
\]
\end{cor}

A natural question at this stage is what other asymptotics can be read off from the above theorem. The following result shows that even (in principle) global functional asymptotics follow from this theorem. 
\begin{cor}\label{cor:meso-spect}
Let $\vA(n)$ denote the adjacency matrix of $\cT(n)${\cg ,} let $\set{\gl_i(n):1\leq i\leq n}$ denote the corresponding eigen-values of $\vA_n$ and let $\hat{F}_n$ denote the corresponding empirical distribution {\cg of eigen-values.} Then under the assumptions of the Theorem~\ref{thm:meso-local}, for the random trees models $\set{\cT(n):n\geq 2} \sim \cL(\gb, \mu, f)$ there exists a deterministic distribution $F_{f,\infty}^{\spec}$ with positive mass at zero such that $\hat{F}_n \convas F_{f,\infty}^{\spec}$. 
\end{cor}

\begin{figure}[htbp]
    \centering
    \begin{subfigure}[b]{0.49\textwidth}
    \includegraphics[width=1.0\linewidth]{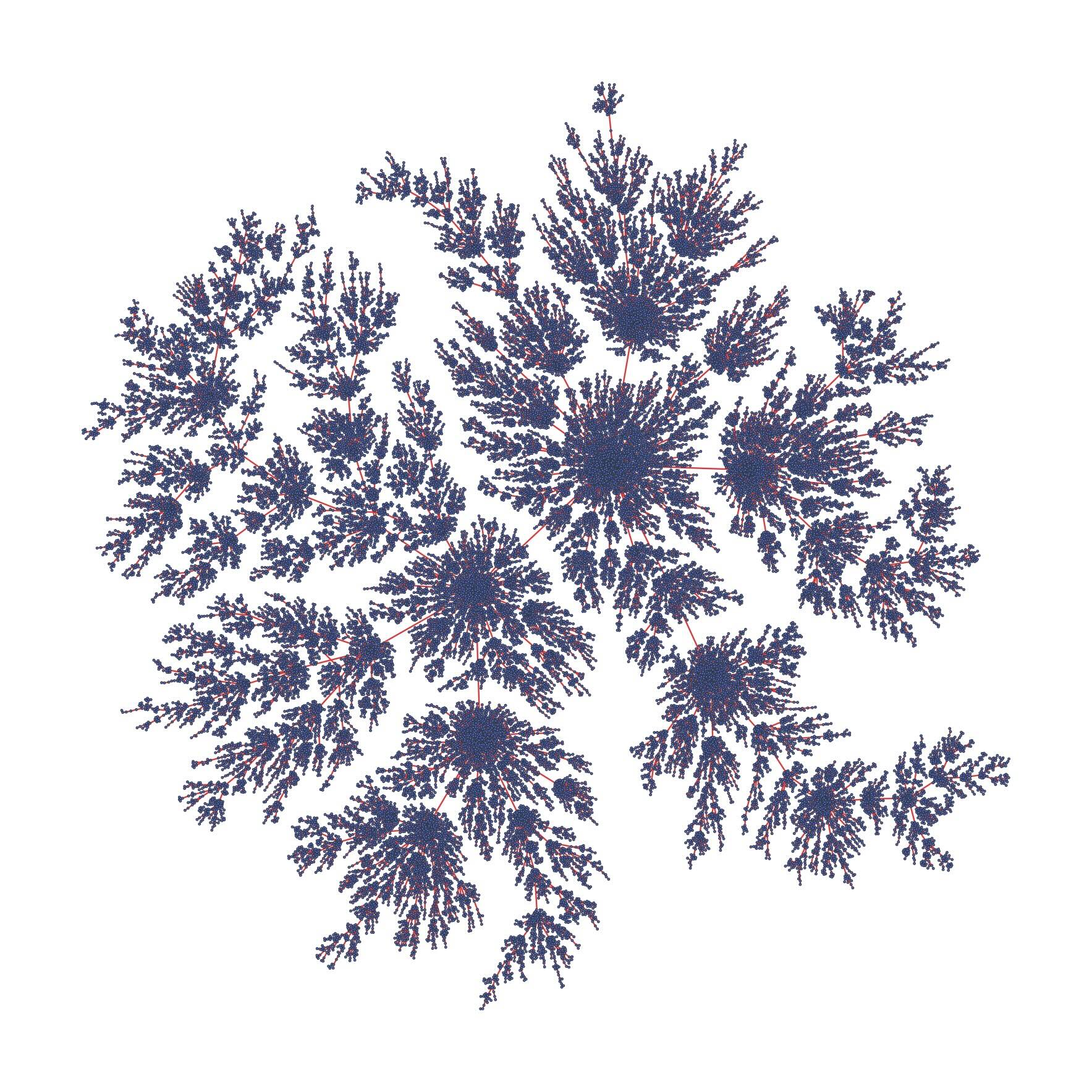}
    \caption{$\xi\equiv0$ (no delay)}
    \end{subfigure}
    \begin{subfigure}[b]{0.49\textwidth}
    \includegraphics[width=1.0\linewidth]{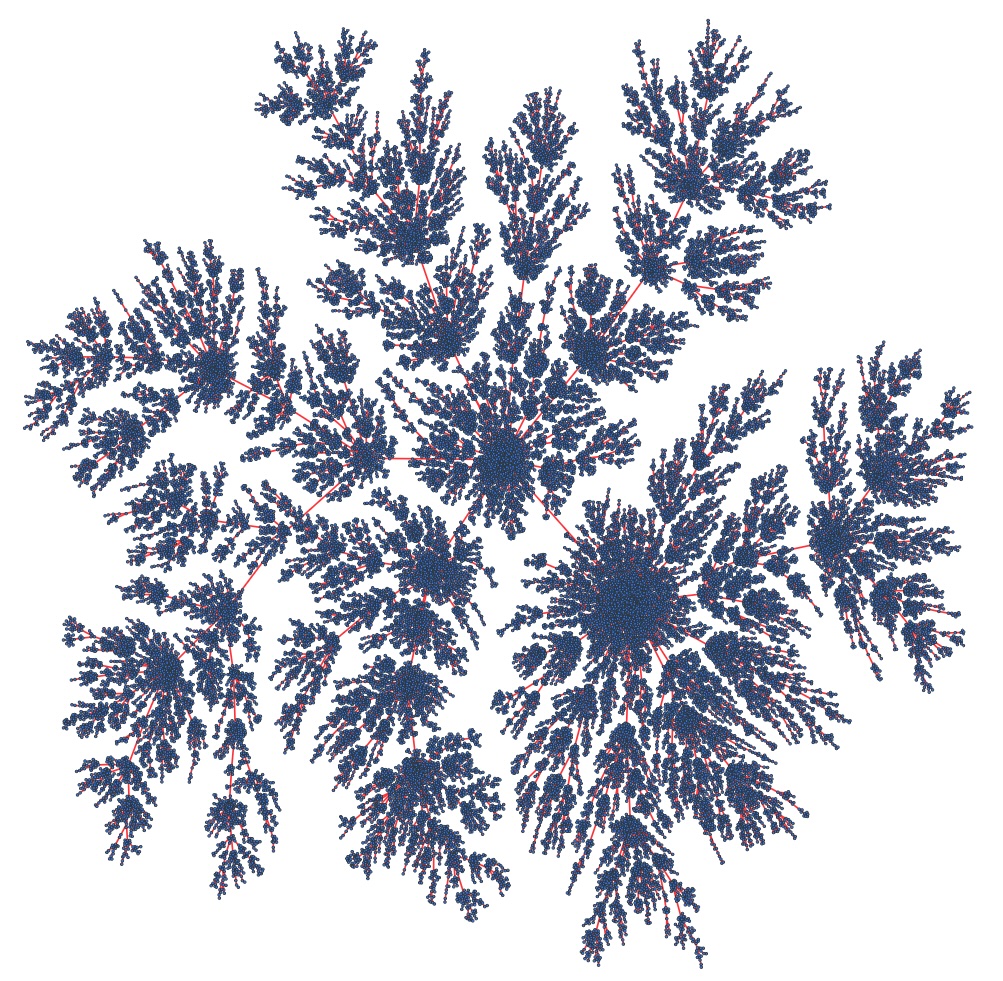}
    \caption{$\xi\sim $Uniform$(0,1)$}
    \end{subfigure}
    \begin{subfigure}[b]{0.49\textwidth}
    \includegraphics[width=1.0\linewidth]{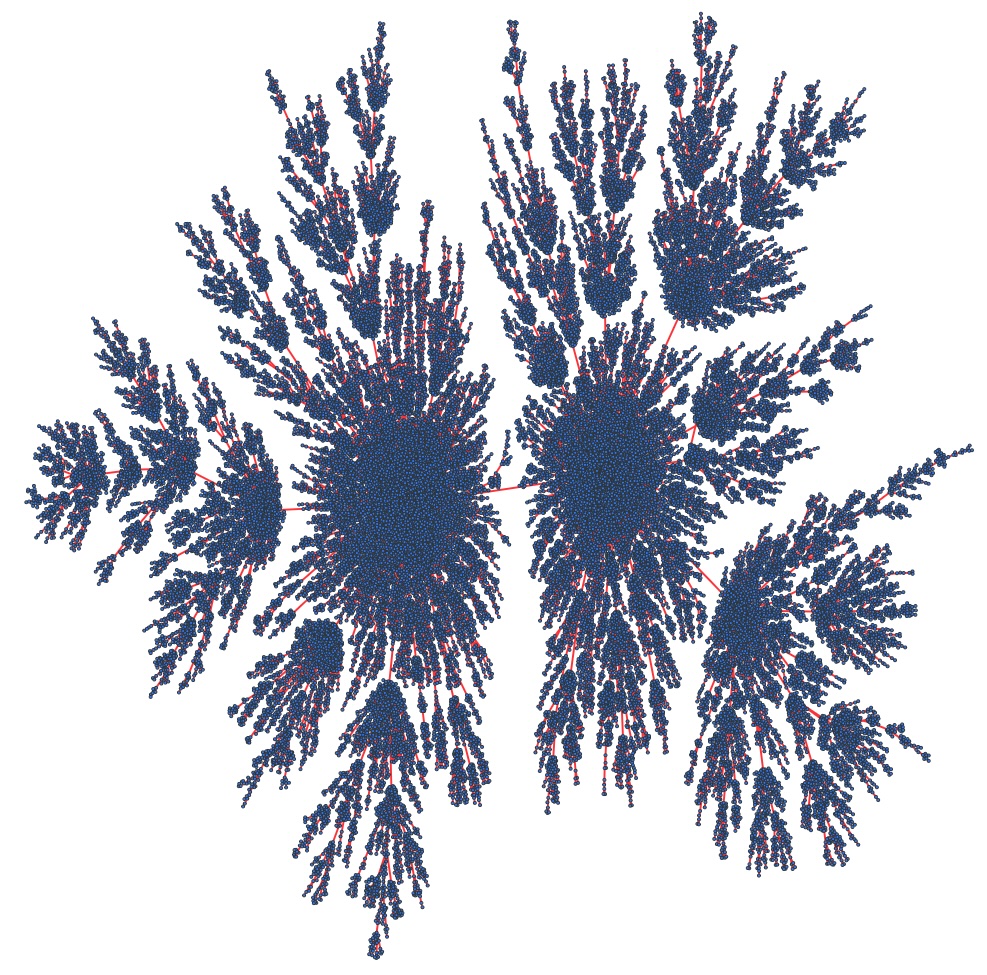}
    \caption{$\xi\sim $1/Uniform$(0,1)$}
    \end{subfigure}
    \begin{subfigure}[b]{0.49\textwidth}
    \includegraphics[width=1.0\linewidth]{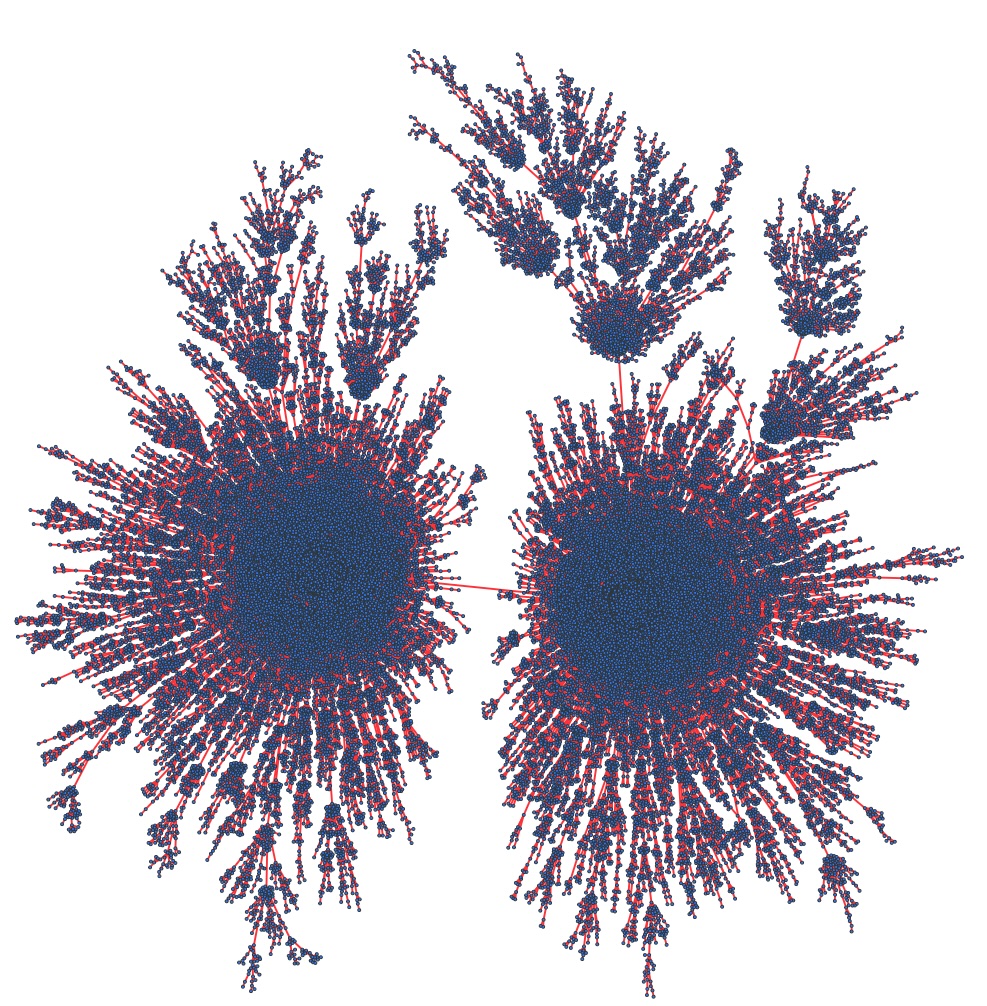}
    \caption{$\xi\sim 1/[\text{Uniform}(0,1)]^{2}$}
    \end{subfigure}
    \caption{Simulated Linear Preferential Attachment trees in the mesoscopic regime with $\beta =1/2$, attachment function taken to be the pure linear preferential attachment function $f(k) = k$ with $\theta=1/2$ and different delay distributions arranged with increasing gradation of mass towards the right tail. Each simulation is on a network size of $n  = 50000$ nodes with different delay distributions. All four settings satisfy the assumptions of Theorem \ref{thm:meso-local} (\ie~\eqref{eqn:815}) and thus have the same local limit as the no delay regime in a large network limit. Delay regime (D) does not satisfy the conditions of Theorem \ref{thm:meso-max-degree}--(i), and thus, for example, the root degree has a different scaling than the setting without delay.}
    \label{fig:sim}
\end{figure}

Since, to first order, in the mesoscopic regime, one does not observe the effect of the delay, a natural question is whether this is observable in second-order fluctuations around limits. We specialize to the linear setting with $f(k) = k+\alpha,~k\geq 1$. Write $\set{p_k:k\geq 1}$ for the corresponding limit degree distribution as in Corollary~\ref{cor:meso-deg}. Resnick and Samorodnitsky~\cite{resnick2015asymptotic} considered the degree counts $\set{N_k(n):k\geq 1}$ in the {\bf no delay regime} and showed that,
\[\set{\sqrt{n}\left(\frac{N_k(n)}{n} - p_k\right):k\geq 1} \convd \set{Z_k:k\geq 1}, \qquad \text{ as } n\to\infty,\]
where $\set{Z_k:k\geq 1}$ is a centered Gaussian process with an appropriate covariance kernel. The proof is initialized by analyzing the number of leaves whose distributional limit,
\[
\sqrt{n}\left(\frac{N_1(n)}{n} - p_1\right) \convd \N\left(0,\frac{(1+\alpha)(2+\alpha)^2}{(3+2\alpha)^2(4+3\alpha)}\right).
\]
This result is then leveraged to prove convergence of degree two counts, which is then used to prove degree three counts, and so on. The following result shows that the effect of the delay is not evident even in the second-order behavior of, for example, leaf counts. In principle, this should lead to convergence of higher degree counts; however, we only deal with the ``base'' case. 

\begin{thm}[Second order fluctuations of degree counts]\label{thm:meso-degree-CLT}
	 Assume that the attachment function is linear with affine parameter $\alpha$, \ie~$f(k) = k+\alpha$, and the delay distribution satisfies Assumption~\ref{ass:lipschitz} and \eqref{eqn:1257}. Then, there exists a sequence of random variables $\set{\fX_n:n\geq 1}$ with $\fX_n/n \convas p_1$ such that $N_1(n)$ satisfies a central limit theorem with 
  \begin{align*}
	     \sqrt{n}\left(\frac{N_1(n)}{n}  -\frac{\fX_n}{n}\right) \convd \N\left(0,\frac{(1+\alpha)(2+\alpha)^2}{(3+2\alpha)^2(4+3\alpha)}\right).
	 \end{align*}
  Furthermore, if the delay distribution satisfies $\lim\limits_{n\to \infty}\sqrt{n} \E\left(\frac{n^\beta\xi\cdot\ind\set{n-n^\beta\xi \geq 1}}{\lfloor n-n^\beta \xi \rfloor}\right) \to 0$, then the centering $\fX_n/n$ can be replaced by $p_1$. 
\end{thm}

\begin{rem}
     When the tails of delay $\xi$ are not unusually heavy, we expect $\E\left(\frac{n^\beta\xi\cdot\ind\set{n-n^\beta\xi \geq 1}}{\lfloor n-n^\beta \xi \rfloor}\right) = o(n^{\beta-1})$. Therefore, $\sqrt{n} \E\left(\frac{n^\beta\xi\cdot\ind\set{n-n^\beta\xi \geq 1}}{\lfloor n-n^\beta \xi \rfloor}\right) = o(n^{\beta - \frac{1}{2}})$, and the condition in Theorem \ref{thm:meso-degree-CLT}, required to have deterministic limit $p_1$ as the centering, is expected to hold for most delay distributions when $\beta < 1/2$. The regime $\beta \geq 1/2$ seems ripe for exploration. As we will see in the next result, functionals like the root degree admit regimes of $\beta$ with potentially new behavior than the setting without delay.   
\end{rem}

The other functionals that could feel the effect of the delay are functions such as the root degree $M(\rho, n)$ or the maximal degree $M_{\sss(1)}(n)$. In the setting with no delay, by~\cite{mori2005maximum}, $n^{-1/(\alpha +2)}\cdot M_{\sss(1)}(n) \convas W_\infty$ for a strictly positive, finite random variable, and the same is true for the root degree. The following result derives conditions for the scaling of the root degree to be affected (or not) by the delay. 

\begin{thm}\label{thm:meso-max-degree}
	Assume that the attachment function $f$ is linear with affine parameter $\alpha$, $\gb <1$, and the delay distribution $\mu$ satisfies the Assumptions in~\ref{ass:lipschitz} \textcolor{black}{and \eqref{eqn:1257}}. Define $X:=\xi^{1/(1-\gb)}$ and $\theta:=1/(2+\ga)$. We have the following behaviors:
 \begin{enumeratei}
     \item If $\E[ X^{1-\theta}] <\infty$, then the rescaled root degree,  
    \[
        M(\rho, n)/n^{\theta} \convas W_\infty \text{ as } n\to\infty,
    \]
for a non-negative finite random variable $W_\infty$, that depends on $\ga,\gb, \xi$. The convergence also holds in $L^2$ and  $\prob(W_\infty >0)>0$.

\item If \textcolor{black}{$\E[X] = \infty$}, then we have for any $\eps>0$,
    \[
    \lim_{n\to\infty}\pr(M(\rho,n) \ge (1-\eps)\E[ X\wedge n] )=1.
    \]
 \end{enumeratei}
\end{thm}

\begin{rem}
    Trying to understand the precise dynamics of the root degree, as well as the maximum degree, is worthy of study just by itself. The above result gives the `first order' behavior of the root degree. We believe that the `second-order' behavior (fluctuations) will lead to several interesting regimes for the delay distribution and associated phase transitions. We are currently investigating such phenomena.
\end{rem}

\begin{rem}
    Note that, in part (ii), if $n^{1-\theta}\pr(X>n)\to\infty$, then $n^{-\theta}\E[ X\wedge n]\to\infty$. \textcolor{black}{In this sense, the root degree scaling undergoes a \emph{phase transition} as the tail of the delay distribution $\prob(X\ge x)$ changes from lighter than $x^{-(1-\theta)}$ to heavier than $x^{-(1-\theta)}$ for large $x$}. One can heuristically explain the transition for the growth order of the root degree as follows. Due to ``large'' delay, there will be $\textsc{extra}_n:=\sum_{i=1}^n \ind\set{\xi_i\ge i^{1-\gb}}$ many extra connections to the root, which is concentrated around its mean $\E(\textsc{extra}_n)= \sum_{i=1}^n \pr(\xi\ge i^{1-\gb}) \approx \E (X\wedge n)$. Suppose $X$ has upper tail behavior given by $\pr(X\ge x)= c x^{-\gc}$ as $x\to\infty$ for some constant $c>0,\gc\in(0,1)$. \textcolor{black}{Then, the number of extra connections is concentrated around $\frac{c}{1-\gc}\cdot n^{1-\gc}$, which is larger than the root degree order with no delay when $\gc<1-\theta$. Due to the preferential attachment property of ``rich getting richer'', for $\gc<1-\theta$, this leads to the root degree growing like $n^{1-\gamma}$.}
    We believe that the correct growth order of the root and maximum degree in case ii) is given by $n^{\theta} \E[(X\wedge n)^{1-\theta}]$. We defer further detailed analysis of such macroscopic functionals to future work. 
\end{rem}

\begin{rem}
    Similar to the case with no delay \cite{mori2005maximum}, we expect a Gaussian CLT for the root degree (and maximum degree) in case i) with centering $n^{\theta}W_\infty$ and scaling $n^{\theta/2}\sqrt{W_\infty}$. In case ii), however, we do not expect any martingale convergence result, but a \textcolor{black}{possibly non-Gaussian} CLT with non-random centering should hold. To keep the exposition simple, we do not pursue this direction here.
\end{rem}

\subsection{Related work}
We placed this work in the general area of dynamic network models in Section~\ref{sec:intro}. The goal of this section is to discuss research that directly influences this paper. Aldous' paper~\cite{aldous-fringe} heralded the now fundamental notion of local weak convergence to understand asymptotics for large discrete random structures; for more recent work on fringe distributions, see the survey~\cite{holmgren2017fringe}. Our proof techniques use stochastic approximation techniques to understand the historical evolution of subtrees, which was first used in the context of preferential attachment models (without delay) in~\cite{rudas2007random}. In the probability community, the study of dynamic network models in the family studied in this paper,  incorporating delay in their evolution, is still primarily in its infancy; the closest papers are~\cites{baccelli2019renewal,king2021fluid,dey2022asymptotic,fk23}. 

\section{Local weak convergence and {\tt sin}-trees}
\label{sec:local-wll}

The goal of this section is to describe standard notions of such convergence, called \emph{local weak convergence}~\cites{aldous-steele-obj,benjamini-schramm,van2023random}. The setting of trees is more straightforward to grasp, and we largely follow Aldous~\cite{aldous-fringe}.

\subsection{Mathematical notation}
We use $\stod$ for stochastic domination between two real-valued probability measures. For $J\geq 1$ let $[J]:= \set{1,2,\ldots, J}$. If $Y$ has an exponential distribution with rate $\gl$, write this as $Y\sim \exp(\gl)$. Write $\bZ$ for the set of integers, $\bR$ for the real line, $\bN$ for the natural numbers and let $\bR_+:=(0,\infty)$. Write $\convas,\convp,\convd$ for convergence almost everywhere, in probability, and in distribution, respectively. For a non-negative function $n\mapsto g(n)$,
we write $f(n)=O(g(n))$ when $|f(n)|/g(n)$ is uniformly bounded, and
$f(n)=o(g(n))$ when $\lim_{n\rightarrow \infty} f(n)/g(n)=0$.
Furthermore, write $f(n)=\Theta(g(n))$ if $f(n)=O(g(n))$ and $g(n)=O(f(n))$.
We write that a sequence of events $(A_n)_{n\geq 1}$
occurs \emph{with high probability} (whp) when $\pr(A_n)\rightarrow 1$ as $n \rightarrow \infty$. One of the core objects of this paper is the study of a sequence of growing random trees $\set{\cT_n:n\geq 1}$. Throughout we will write $\deg(v,\cT_n)$ for the degree of the vertex $v$ in tree $\cT_n$ and write for the empirical degree counts,  
\begin{equation}
    \label{eqn:deg-count}
    N_k(n) = \sum_{v\in \cT_n} \ind\set{\deg(v,\cT_n) = k }, \qquad k\geq 1.
\end{equation}
When there is no scope for confusion, we will simplify {\cg $\deg(v,\cT_n)$} to $\deg(v,n)$. 
\subsection{Fringe decomposition for trees}
 Recall the space of finite rooted trees  $\bbT$  described in Section \ref{sec:ctbp-def}.  For any $r\geq 0$ and $\bt\in \bbT$, let $B(\bt, r) \in \bbT$ denote the subgraph of $\bt$ of vertices within graph distance $r$ from $\rho_{\bt}$, viewed as an element of $\bbT$ and rooted again at $\rho_{\bt}$. 

 Given two rooted finite trees $\bs, \bt \in \bbT$,  say that $\bs \simeq \bt $ if  there exists a {\bf root
preserving} isomorphism between the two trees viewed as unlabelled graphs. Given two rooted trees $\bt,\bs \in \bbT$ (\cite{benjamini-schramm},~\cite{van2023random}*{Equation 2.3.15}), define the distance 
\begin{align}
\label{eqn:distance-trees}
	d_{\bbT}(\bt,\bs):= \frac{1}{1+R^*}, \qquad \text{with } \qquad R^* =\sup\{r: B(\bt, r) \simeq B(\bs, r)
	\}.
\end{align}

Next, fix a tree $\bt\in \bbT$ with root $\rho = \rho_\bt$ and a vertex $v\in \bt$ at (graph) distance $h$ from the root.  Let $(v_0 =v, v_1, v_2, \ldots, v_h = \rho)$ be the unique path from $v$ to $\rho$. The tree $\bt$  can be decomposed as $h+1$ rooted trees $f_0(v,\bt), \ldots, f_h(v,\bt)$, where $f_0(v,\bt)$ is the tree rooted at $v$, consisting of all vertices for which there exists a path from the root passing through $v$. For $i \ge 1$, $f_i(v,\bt)$ is the subtree rooted at $v_i$, consisting of all vertices for which the path from the root passes through $v_i$ but not through $v_{i-1}$. 

{\cg Let $\bbT^\infty := \bbT^{\mathbb{N}}$. For $\bt \in \bbT$ and $v\in \bt$, we define $F(v,\bt) \in \bbT^\infty$ as} 
\[F(v, \bt) = \left(f_0(v,\bt), f_1(v,\bt) , \ldots, f_h(v,\bt), \emptyset, \emptyset, \ldots \right),\]
{\cg and call it the} fringe decomposition of $\bt$ about the vertex $v$. Call $f_0(v,\bt)$ the {\bf fringe} of the tree $\bt$ at $v$.
For $k\geq 0$, call $F_k(v,\bt) = (f_0(v,\bt) , \ldots, f_k(v,\bt))$ the {\bf extended fringe} of the tree $\bt$ at $v$ truncated at distance $k$ from $v$ on the path to the root.

Now consider the space $\bbT^\infty$. The metric in~\eqref{eqn:distance-trees} extends  to $\bbT^\infty$, \eg\ using the distance,
\begin{align}
\label{eqn:dist-inf}
	d_{\bbT^\infty}((\bt_0, \bt_1, \ldots),(\bs_0, \bs_1, \ldots)):= \sum_{i=0}^\infty \frac{1}{2^i} d_{\bbT}(\bt_i, \bs_i). 
\end{align}
We can also define analogous extensions to $\bT^k$ for finite $k$.  

Next, an element $\bfomega = (\bt_0, \bt_1, \ldots) \in \bbT^\infty$, with $|\bt_i|\geq 1$ for all $ i\geq 0$,  can be thought of as a locally finite infinite rooted tree with a {\bf s}ingle path to {\bf in}finity (thus called a {\cg{\tt sin}-tree~\cite[Section 4.1]{aldous-fringe}}), as follows: Identify the sequence of roots of $\set{\bt_i:i\geq 0}$ with the integer lattice $\Zbold_+ = \set{0,1,2,\ldots}$, {\cg rooted at $\rho = 0$ and equipped with nearest neighbor edges connecting the root of $\bt_i$ with that of $\bt_{i+1}$ for $i \ge 0$}. Analogous to the definition of extended fringes for finite trees, for any $k\geq 0$, write 
$F_k(0,\bfomega)= (\bt_0, \bt_1, \ldots, \bt_k)$. 



Call this the extended fringe of the tree $\bfomega$ at vertex $0$, till distance $k$, on the infinite path from $0$. Call $\bt_0 = F_0(0,\bfomega)$ the {\bf fringe} of the {\tt sin}-tree $\bfomega$. Now suppose $\prob$ is a probability measure on $\bbT^\infty$ such that, for $\TT:= (\bt_0(\TT), \bt_1(\TT),\ldots)\sim \prob$,  $|\bt_i(\TT)|\geq 1$ a.s.  $\forall~i\geq 0$. Then $\TT$ can be thought of as an infinite {\bf random} {\tt sin}-tree. 

Define a matrix $\vQ = (\vQ(\vs,\vt): \vs, \vt \in \bbT)$ as follows: suppose the root $\rho_{\vs}$ in $\vs$ has degree $\deg(\rho_{\vs}) \ge 1$, and let $(v_1,\ldots, v_{\deg(\rho_{\vs})})$ denote its children. For $1\leq i\leq v_{\deg(\rho_{\vs})}$, let $f(\vs, v_i)$ be the subtree below $v_i$ and rooted at $v_i$, viewed {\cg as} an element of $\bbT$. Write,
\begin{align}
\label{eqn:Q-def}
	\vQ(\vs,\vt):= \sum_{i=1}^{\deg(\rho_{\vs})} \ind\set{d_{\bbT}(f(\vs, v_i), \vt) = 0}. 
\end{align} 
Thus, $Q(\vs, \vt)$ counts the number of descendant subtrees of the root of $\vs$ that are isomorphic in the topological sense to $\vt$. If $\deg(\rho_{\vs})=0$, define $Q(\vs, \vt)=0$. Now consider a sequence $(\bar \vt_0, \bar \vt_1, \dots)$ of trees in $\bbT$ such that $Q(\bar \vt_i, \bar \vt_{i-1}) \ge 1$ for all $i \ge 1$. Then there exists a unique infinite {\tt sin}-tree $\TT$ with infinite path indexed by $\Zbold_+$ {\cg constructed by identifying the root of $\bar \vt_i$ with $i$ for all $i \in \Zbold_+$}. {\cg Conversely, any infinite {\tt sin}-tree $(\vt_0, \vt_1,\dots) \in \bbT^{\infty}$ has such a representation, which is verified by taking $\bar \vt_i$ to be the union of (vertices and induced edges) of $\vt_0,\dots, \vt_i$, {\cg rooted at the root of $\vt_i$}, for each $i \in \Zbold_+$.} {\cg Following~\cite[Section 4.1]{aldous-fringe}}, we call this the \emph{monotone representation} of the {\tt sin}-tree $\TT$.

\subsubsection{Convergence on the space of trees}
\label{sec:fringe-convg-def}
For $1\leq k\leq \infty$, let $\cM_{\pr}(\bbT^k)$ denote the space of probability measures on the associated space, metrized using the topology of weak convergence inherited from the corresponding metric on the space $\bbT^k$, see, e.g.,~\cite{billingsley2013convergence}. 
Suppose $\set{\cT_n}_{n\geq 1} \subseteq \bbT$ be a sequence of {\bf finite} rooted random trees on some common probability space (for notational convenience, assume $|\cT_n|= n$, or more generally $|\cT_n|\convas \infty$). For $n\geq 1$ and for each fixed $k\geq 0$,  the empirical distribution of fringes up to distance $k$ {\cg is defined as}
\begin{align}
\label{eqn:empirical-fringe-def}
	\fP_{n}^k:= \frac{1}{n} \sum_{v\in \cT_n} \delta\set{F_k(v,\cT_n)} 
\end{align} {\cg where for $\vt^{(k)} \in \bbT^k$, $\delta\set{\vt^{(k)}}$ denotes the Dirac measure on $\cM_{\pr}(\bbT^k)$ assigning mass $1$ at $\vt^{(k)}$. }Thus, $\set{\fP_{n}^k:n\geq 1}$ can be viewed as a random sequence in  $\cM_{\pr}(\bbT^k)$. In the following, $\E[\fP_n^{0}]$ denotes the measure given by $\E[\fP_n^{0}](\vt):= \E[\fP_n^{0}(\vt)], \, \vt \in \bbT$.

\begin{defn}[Local weak convergence]
	\label{def:local-weak}
 Fix a probability measure $\varpi$ on $\bT$.
	\begin{enumeratea}
 \item \label{it:fringe-exp} Say that a sequence of trees $\set{\cT_n}_{n\geq 1}$ converges in {\bf expectation}, in the fringe sense, to $\varpi$, if \[\E[\fP_n^{0}] \to \varpi, \quad \text{ as } n\to\infty. \]
 Denote this convergence by $\TT_n\Efr \varpi$ as $n\to\infty$.
	    \item \label{it:fringe-a}  Say that a sequence of trees $\set{\cT_n}_{n\geq 1}$ converges in the probability sense, in the fringe sense, to $\varpi$, if \[\fP_n^{0} \probc \varpi, \quad \text{ as } n\to\infty. \]
	Denote this convergence by $\TT_n\probfr \varpi$ as $n\to\infty$.
	    \item \label{it:fringe-b} Say that a sequence of trees $\set{\cT_n}_{n\geq 1}$ converges in probability, in the {\bf extended fringe sense}, to a limiting infinite random {\tt sin}-tree $\TT_{\infty}$ if for all $k\geq 0$ one has
	  \[\fP_n^k \probc \prob\left(F_k(0,\TT_{\infty}) \in \cdot \right), \qquad \text{ as } n\to\infty. \]
	Denote this convergence by $\TT_n\probcrf \TT_{\infty}$ as $n\to\infty$.
	\end{enumeratea}
\end{defn} 
In an identical fashion, one can define notions of convergence in distribution or almost surely in the fringe, respectively, in the extended fringe sense. 
Letting $\varpi_{\infty}(\cdot) = \pr(F_0(0, \cT_{\infty}) = \cdot)$ denote the distribution of the fringe of $\cT_\infty$ on $\bT$, convergence in (c) above clearly implies convergence in notion (b) with $\varpi =\varpi_{\infty}(\cdot) $.
If the limiting distribution $\varpi$ in (b) has a certain `stationarity' property (defined below), \emph{convergence in the fringe sense implies convergence in the extended fringe sense} as we now describe. Under an ``extremality'' condition of the limit objection in (a), convergence in expectation implies convergence in probability as in (b). We need the following definitions.

\begin{defn}[{\cg Fringe distribution~\cite[Section 2.1]{aldous-fringe}}] \label{fringedef}
	Say that a probability measure $\varpi$ on $\bbT$ is a fringe distribution if 
	\[\sum_{\vs} \varpi(\vs) \vQ(\vs, \vt) = \varpi(\vt), \qquad \forall~\vt \in \bbT. \]
\end{defn}
It is easy to check that the space of fringe distributions $\cM_{\pr, \fringe}(\bbT) \subseteq \cM_{\pr}(\bT)$ is a convex subspace of the space of probability measure on $\bT$ and thus one can talk about extreme points of this convex subspace. The following fundamental theorem is one of the highlights of~\cite{aldous-fringe}. 
\begin{thm}[{\cg\cite[Proposition 7]{aldous-fringe}}]
\label{thm:aldous-efr-pfr}
    Fix a fringe distribution $\varpi \in \cM_{\pr, \fringe}(\bbT)$. Suppose a sequence of trees $\set{\cT_n:n\geq 1}$ converges in the expected fringe sense $\cT_n \Efr \varpi$ as $n\to\infty$. If $\varpi$ is extremal in the space of fringe measures, then the above convergence in expectation automatically implies $\cT_n \probfr \varpi$. 
\end{thm}
The advantage of this result is that for proving convergence in the probability fringe sense, at least under the extremality of the limit object, dealing with expectations is enough. 
The next result shows that convergence in the probability fringe sense often automatically implies convergence to a limit infinite {\tt sin}-tree. We need one additional definition. For any fringe distribution $\varpi$ on $\bbT$, one can uniquely obtain the law $\varpi^{EF}$ of a random {\tt sin}-tree $\TT$ with monotone decomposition $(\bar\bt_0(\TT), \bar\bt_1(\TT),\ldots)$ such that for any $i \in \Zbold_+$, any $\bar \vt_0, \bar \vt_1, \dots$ in $\bbT$,
\begin{align}\label{ftoef}
\varpi^{EF}((\bar\bt_0(\TT), \bar\bt_1(\TT),\ldots, \bar\bt_i(\TT)) = (\bar \vt_0,\vt_1,\dots,\vt_i)) := \varpi(\vt_i) \prod_{j=1}^{i}Q(\vt_i,\vt_{i-1}),
\end{align}
where the product is taken to be one if $i=0$. The following Lemma follows by adapting the proof of~\cite{aldous-fringe}*{Propositions 10 and 11}, and the proof is omitted.

\begin{lem}\label{ftoeflemma}
Suppose a sequence of trees $\set{\cT_n}_{n\geq 1}$ converges in probability, in the fringe sense, to $\varpi$. Moreover, suppose that $\varpi$ is a fringe distribution in the sense of Definition~\ref{fringedef}. Then $\set{\cT_n}_{n\geq 1}$ converges in probability, in the extended fringe sense, to a limiting infinite random sin-tree $\TT_{\infty}$ whose law $\varpi^{EF}$ is uniquely obtained from $\varpi$ via~\eqref{ftoef}.
\end{lem}

Fringe convergence and extended fringe convergence imply convergence of functionals, such as the degree distribution. For example, letting $\cT_{\varpi} \sim \varpi$ with root denoted by $0$ say, convergence in notion~\eqref{it:fringe-a} in particular implies that for any $k\geq 0$, 
\begin{align}
\label{eqn:deg-convg-fr}
	\frac1n\cdot\#\set{v\in \TT_n: \deg(v) = k+1} \convp \prob(\deg(0,\cT_{\varpi})=k).
\end{align}
However, both convergences give more information about the asymptotic properties of $\set{\cT_n:n\geq 1}$ beyond its degree distribution.

\section{Proofs}
\label{sec:proofs-meso}

\noindent {\bf Overview of the proofs:} The section contains proofs of all our main results described in Section~\ref{sec:main-res}.  We start by showing that the degree distribution converges, namely Corollary~\ref{cor:meso-deg} in Section~\ref{sec:pf-corr-deg-meso} (and also prove Lemma~\ref{lem:815} on checkable conditions on the delay distribution). Degree distribution convergence, surprisingly, can be leveraged, using stochastic approximation techniques, to show local weak convergence, leading to the proof of Theorem~\ref{thm:meso-local} in Section~\ref{sec:pf-thm-lwc-meso}. The remaining subsections contain proofs of the second-order fluctuations of the degree distribution (Theorem~\ref{thm:meso-degree-CLT})  and root degree evolution (Theorem~\ref{thm:meso-max-degree}), respectively, in the affine attachment setting. 

\subsection{Proofs of Lemma~\ref{lem:815} and Corollary~\ref{cor:meso-deg}:} 
\label{sec:pf-corr-deg-meso}

Throughout this section, for the attachment function $f$ and the delay distribution $\xi$, we will work under Assumption~\ref{ass:attach-func} (which, to some extent, is more relevant for the local weak convergence) and Assumption~\ref{ass:lipschitz} (which plays an essential role in this section). We will use $\set{\cF_n:n\geq 1}$ for the natural filtration of the process. We write $\E_\xi(.) = \E_{\xi}(. \mid \cF_n)$ to simplify notation when there is no scope for confusion.

\begin{proof}[Proof of Lemma~\ref{lem:815}:]
(a) First, we reduce the general $\gb\in[0,1)$ case to the $\gb=0$ case by considering the random variable $X:=\xi^{\frac{1}{1-\gb}}$. Note that, $\lfloor x\rfloor \ge x/2$  and $x^{1-\gb}-1=(1-\gb)\int_1^x t^{-\gb}dt\ge (1-\gb) x^{-\gb}(x-1)$ for $x\ge 1$. Moreover, $n^\gb X^{1-\gb}=n^\gb \xi\le n-1$ implies that $X^{1-\gb} \le n^{1-\gb}(1-1/n)$ or $X\le n(1-1/n)=n-1$. Thus, we get
\begin{align}
    \E\left[\frac{n^{\gb}\xi \cdot \ind{\{ n -n^{\gb}\xi \geq 1\}}}{\fplus{n-n^{\gb}\xi}}\right]
    &\le \frac{2}{1-\gb} \E\left[\frac{n^\gb X^{1-\gb} \cdot\ind{\{X \leq n-1\}}}{n-X}\right]\notag\\
    &\le \frac{4}{1-\gb} \E\left[ (X/n)^{1-\gb}\cdot \ind{\{X \leq n/2\}}\right] \nonumber\\
    &\hspace{2cm}+ \frac{4}{1-\gb} \E\left[\frac{X}{n-X}\cdot \ind{\{n/2\leq X \leq n-1\}}\right].\label{eq:pf815}
\end{align}
The first term converges to zero by DCT, and using $X\leq \lceil X\rceil$, we get 
\begin{align*}
    \E\left[\frac{X}{n-X}\cdot \ind{\{n/2\leq X \leq n-1\}}\right] 
    &\le \E\left[\frac{\lceil X\rceil}{n-\lceil X\rceil}\cdot \ind{\{n/2\leq \lceil X\rceil \leq n-1\}}\right]\\
    &=\sum_{k=1}^{n/2} \frac1k \cdot (n-k)\pr(\lceil X\rceil=n-k)\\
    &\le (\log n)\cdot \sup_{k\ge n/2}k\pr(\lceil X\rceil=k).
\end{align*}
The last term converges to zero by our assumption that $\lim_{n\to\infty}n\log n \cdot \pr(\lceil X \rceil =n)=0$.

(b) Note that for a fixed integer $\ell\ge 1$, we have
\begin{align*}
    \sum_{n=1}^\infty \frac1n\cdot \frac{\ell}{n-\ell}\cdot \ind{\{n-\ell\ge 1\}}
    &= \sum_{n=1+\ell}^\infty \left(\frac{1}{n-\ell}-\frac1n\right) = \sum_{n=1}^\ell\frac1n \le \log (1+\ell).
\end{align*}
For general $\gb\in[0,1)$ case, we use equation~\eqref{eq:pf815} with $X=\xi^\frac{1}{1-\gb}$, to get
\begin{align*}
    &(1-\gb)\cdot \sum_{n=1}^\infty \frac{1}{n}\E\left[\frac{n^{\gb}\xi \ \ind{\{ n -n^{\gb}\xi \geq 1\}}}{\fplus{n-n^{\gb}\xi}}\right] \\
    &\qquad\qquad\le \sum_{n=1}^\infty \E\left[ \frac{X^{1-\gb}}{n^{2-\gb}} \ind{\{n\geq 2X\}}\right] + 4\sum_{n=1}^\infty \frac{1}{n} \E\left[\frac{X}{n-X}\cdot \ind{\{n-X \geq 1\}}\right].
\end{align*}
Both the terms are finite by our assumptions that $\gb<1$ and $\E\log_+ \xi<\infty$.
\end{proof}

\noindent {\bf Proof of Corollary~\ref{cor:meso-deg}:} 
For the rest of the proof, we switch notation from~\eqref{eqn:deg-count} and for $k\geq 1$, and network size $n$, write $Z_k(n)$ be the number of vertices of degree $k$ at time $n$ (instead of $N_k(n)$ which will provide more flexible notation to count other functionals later such as number of fringe trees of a specific shape). Let,
\begin{align}
\label{eqn:psi-def}
    \Psi(n) = \sum_{v\in \cT(n)} f(\deg(v,n)) = \sum_{k=1}^\infty f(k) Z_k(n).
\end{align}

For vertices $i\in [n]$, define \begin{align*}
    w_1(i,n) &= \frac{f(\deg(i,n))}{\Psi(n)}\pr(n-n^{\gb}\xi < i),\\
    w_2(i,n) &= \frac{f(\deg(i,n))}{\Psi(n)}\E_\xi\left(\frac{\ind\set{n-n^{\gb}\xi \geq i}}{\Psi(\lfloor n - n^{\gb} \xi \rfloor)} \cdot \left[\Psi(\lfloor n - n^{\gb} \xi \rfloor)-\Psi(n)\right]\right),\\
    w_3(i,n) &= \E_\xi\left(\frac{\ind\set{n-n^{\gb}\xi \geq i}}{\Psi(\lfloor n - n^{\gb} \xi \rfloor)}\cdot\left[f(\deg(i,n) - f(\deg(i,\lfloor n - n^{\gb} \xi \rfloor)\right]\right),\\
    w_4(i,n) &= -\ind\set{i=1}\pr(\xi > n^{1-\beta}),
\end{align*}
where the final object is required to take into account large delays that automatically lead to new vertices connecting to the root, namely, vertex $i=1$. For any vertex $i \in [n]$, we have,
\begin{align}
\label{eqn:prob-decomp}
    &\pr\left( \text{vertex } n+1 \text{ attaches to vertex } i\mid \cF_n\right)\nonumber \\
    &\hspace{5mm}= \E_\xi \left(\frac{\ind\set{n-n^\beta\xi \geq i}}{\Psi(\lfloor n - n^{\gb} \xi \rfloor)} f(\deg(i,\lfloor n - n^{\gb} \xi \rfloor))\right) + \ind\set{i=1} \pr(n-n^\beta \xi < 1)\notag\\
    &\hspace{5mm}=\frac{f(\deg(i,n))}{\Psi(n)} - \left[\sum_{j=1}^4 w_j(i,n)\right]. 
\end{align}
\begin{lemma}
    \label{lem: deg-approx}
    For any $k\geq 1$, there exists a sequence of random variables $\set{\pi_k(n)}_{n\geq 1}$ such that for all $n\geq 1$, we have
    \begin{align}
        \E(Z_1(n+1) - Z_1(n)\mid \cF_n) &= 1 - f(1)\frac{Z_1(n)}{\Psi(n)} + \pi_1(n), \label{eqn:deg-approx1} \\
        \E(Z_k(n+1) - Z_k(n)\mid \cF_n) &= f(k-1)\frac{Z_{k-1}(n)}{\Psi(n)} - f(k)\frac{Z_k(n)}{\Psi(n)} + \pi_k(n) \qquad\text{for } k\geq 2 \label{eqn:deg-approxk}
    \end{align}
    with $\lim_{n\to \infty}\pi_k(n) \to 0$ almost surely and in $\cL^1$ for all $k\geq 1$.
\end{lemma}
\begin{proof}
    Assume $f(0) = 0$ and let $\deg(i,n) = 0$ for $n < i$. We prove~\eqref{eqn:deg-approx1}, the same argument works for $k\geq 2$, namely~\eqref{eqn:deg-approxk}, by recursively leveraging the decomposition in~\eqref{eqn:prob-decomp} and inductively building up with $\pi_j(n) \convas 0$ for $j < k$. \medskip
    
    \noindent \textbf{When $f$ is Lipschitz:} with Lipschitz constant $L$, for any $m,n \geq 1$, we have
\begin{align}
\label{eqn:denom-lipschitz}
    \left|\Psi(n+m) -\Psi(n)\right| &= \sum_{i=1}^{n+m} \left|f(\deg(i,n+m)) -f(\deg(i,n))\right| \nonumber\\
    &\leq \sum_{i=1}^{n+m} L\left[\deg(i,n+m)-\deg(i,n)\right] = 2Lm.
\end{align}
Since $\inf_k f(k) = f_* > 0 $, we also have  $\Psi(n) \geq nf_*$. Next, let 
\[
W_j(n) = \sum_{i=1}^n \ind\set{\deg(i,n)=1}w_j(i,n),\text{ for } 1\leq j \leq 4.
\]
We have 
 \begin{align*}
            |W_1(n)|= \sum_{i=1}^n \ind\set{\deg(i,n) =1} w_1(i,n) 
            &\leq  \frac{f(1)}{nf_*} \E(n^{\gb} \xi \cdot \ind\set{n-n^{\gb} \xi\geq 1}) \\
            &\leq \frac{f(1)}{f_*}\E\left(n^{\gb-1}\xi\cdot\ind\set{n-n^{\gb} \xi\geq 1} \right),
        \end{align*}
    Since $n^{\gb-1}\xi\ind\set{n-n^{\gb} \xi\geq 1} \leq 1$, by DCT, we have $W_1(n)\convas 0$. Using~\eqref{eqn:denom-lipschitz}, we have 
 \begin{align*}
            |W_2(n)| &\leq \sum_{i=1}^n \ind\set{\deg(i,n) =1} |w_2(i,n)| \leq \frac{f(1)}{nf_*^2}\sum_{i=1}^n \E\left(\frac{\ind\set{n-n^{\gb}\xi \geq i}}{\lfloor n-n^{\gb}\xi\rfloor}\cdot2Ln^{\gb}\xi\right) \\
            &\leq \frac{2Lf(1)}{f_*^2}\E\left(\frac{n^{2\gb-1} \xi^2 \cdot\ind\set{n-n^{\gb}\xi \geq 1}}{\lfloor n-n^{\gb}\xi\rfloor} \right) \leq \frac{2Lf(1)}{f_*^2}\E\left(\frac{n^{\gb} \xi\cdot\ind\set{n-n^{\gb}\xi \geq 1}}{\lfloor n-n^{\gb}\xi\rfloor} \right)
        \end{align*}
Therefore under assumptions~\ref{ass:lipschitz}, we have $W_2(n)\convas 0$. Also
    \begin{align*}
        W_3(n) &=  \sum_{i \leq n}\ind\set{\deg(i,n) =1} w_3(i,n) = 0\quad\text{ and }\quad
        |W_4(n)| \leq \pr(\xi\geq n^{\beta-1}). 
    \end{align*}
Since $\pr(\xi < \infty)= 1$, we have $W_4(n)\to 0$. 
Define $$\pi_1(n) := -\left[W_1(n) + W_2(n) + W_3(n) + W_4(n)\right].$$ Note that all the upper bounds for $W_j(n)$ for $1\leq j \leq 4$ are deterministic and converge to $0$, hence $\pi_1(n) \to 0$ almost surely and in $\cL^1$.
    We have, by~\eqref{eqn:prob-decomp}
    \begin{align*}
        \E(Z_1(n+1)-Z_1(n)\mid \cF_n) &= 1 - \sum_{i=1}^n \ind\set{\deg(i,n) =1} \pr( n+1 \text{ attaches to } i\mid \cF_n)\\
     &= 1-f(1)\frac{Z_1(n)}{\Psi(n)} + \pi_1(n).
    \end{align*}

\noindent \textbf{When $f$ is sublinear and non-decreasing:} Note that the Lipschitz property of $f$ is used to bound $W_2(n)$. When $f$ is non-decreasing, $\Psi(n)$ is non-decreasing with $n$. Hence we have \begin{align*}
    \frac{|\Psi(n)-\Psi(\lfloor n-n^\beta\xi \rfloor)|}{\Psi(n)} \leq 1,
\end{align*} and therefore\begin{align}\label{eqn:special-special}
    |W_2(n)| \leq \frac{f(1)}{f_*} \sum\limits_{i=1}^n \E\left(\frac{\ind\set{n-n^\beta \xi \geq i}}{\lfloor n-n^\beta \xi\rfloor}\right) \leq \frac{f(1)}{f_*}\E\left(\frac{n^\beta \xi\cdot \ind\set{n-n^\beta \xi \geq 1}}{\lfloor n-n^\beta \xi\rfloor}\right) \to 0.
\end{align}
The rest of the argument is the same as above. Thus, when $f$ is increasing, the Lipschitz property is not required for the above argument to work. 

This proves the result for $k=1$. A similar argument for $k\geq 2$ yields~\eqref{eqn:deg-approxk}.
\end{proof}

We now complete the proof of Corollary~\ref{cor:meso-deg} using the above Lemma \ref{lem: deg-approx}. We handle the sublinear and linear attachment regimes separately.\\

\noindent\textbf{Sublinear Regime:}
Recall the normalizing factor $\Psi(n)$ defined in \eqref{eqn:psi-def}. For $n\in \bN$, we define the following collection of function-valued processes $\set{X_k^n(\cdot), k\geq 1}$ and $\Psi^n(\cdot)$ where for all $t\geq 0$
\begin{align*}
    X_k^n(t) &= \frac{Z_k(\lfloor nt \rfloor)}{n} + \frac{nt-\lfloor nt \rfloor}{n} \left(Z_k(\lceil nt \rceil) - Z_k(\lfloor nt \rfloor)\right), \ k\ge 1,\text{ and}\\
    \Psi^n(t) &= \frac{\Psi(\lfloor nt \rfloor)}{n} + \frac{nt-\lfloor nt \rfloor}{n} \left(\Psi(\lceil nt \rceil) - \Psi(\lfloor nt \rfloor) \right).
\end{align*}
For $k\geq 1$, we write
\begin{align*}
    X_k^n(t) = \frac{M_k(\lfloor nt \rfloor)}{n}  + \frac{1}{n}\left[\sum_{i=1}^{\lfloor nt \rfloor -1} \E(Z_k(i+1) - Z_k(i)\mid \cF_i)\right]
\end{align*}where $$M_k(l) = \sum_{i=1}^l [Z_k(i) - \E(Z_k(i)\mid \cF_i)],\qquad \ell\ge 1$$ is a $\cL^2-$martingale with bounded difference (difference being at most $2$). Moreover, the quadratic variation of the above martingale $M_k(l)$ is bounded above by $4l$. Therefore, by Doob's $\cL^2$ inequality, for every $T >0$ we have \begin{align}\label{eq:mg-bound}
    \pr\left(\sup\limits_{t\in [0,T]}\frac{|M_k(\lfloor nt\rfloor)|}{n} \geq \epsilon\right)\leq \frac{\E(\left< M_k(\lfloor nT\rfloor)\right>)}{n^2\epsilon^2} \leq \frac{4T}{n\epsilon^2} \to 0.
\end{align}
By Lemma~\ref{lem: deg-approx}, we have 
\begin{align*}
    Z_1(\lfloor nt \rfloor) &= \frac{M_1(\lfloor nt \rfloor)}{n}  + \frac{1}{n}\sum_{i=1}^{\lfloor nt \rfloor -1} \left[1 - f(k)\frac{Z_k(i)}{\Psi(i)} \right]+ \frac{1}{n} \sum\limits_{i=1}^{\lfloor nt \rfloor -1} \pi_1(i),\\
    Z_k(\lfloor nt \rfloor) &= \frac{M_k(\lfloor nt \rfloor)}{n}  + \frac{1}{n}\sum_{i=1}^{\lfloor nt \rfloor -1} \left[ f(k-1)\frac{Z_{k-1}(i)}{\Psi(i)} - f(k)\frac{Z_k(i)}{\Psi(i)} \right] + \frac{1}{n} \sum\limits_{i=1}^{\lfloor nt \rfloor -1} \pi_k(i). 
\end{align*}

Now rewrite this as
\begin{align}
    Z_1(\lfloor nt \rfloor) &= \frac{M_k(\lfloor nt \rfloor)}{n}  + \int_{0}^t \left[1 - f(k)\frac{Z_k(\lfloor nx \rfloor)}{\Psi(\lfloor nx \rfloor)} \right] dx + \frac{1}{n} \sum\limits_{i=1}^{\lfloor nt \rfloor -1} \pi_1(i),\label{eqn:integral-mart1}\\
    Z_k(\lfloor nt \rfloor) &= \frac{M_k(\lfloor nt \rfloor)}{n}  + \int_{0}^t \left[f(k-1)\frac{Z_{k-1}(\lfloor nx \rfloor)}{\Psi(\lfloor nx \rfloor)} - f(k)\frac{Z_k(\lfloor nx \rfloor)}{\Psi(\lfloor nx \rfloor)} \right] dx + \frac{1}{n} \sum\limits_{i=1}^{\lfloor nt \rfloor -1} \pi_k(i).\label{eqn:integral-martk}
\end{align}

Also, note that, by Lemma~\ref{lem: deg-approx}, for any $T\geq 0$ and $k\geq 1$, as $n\to \infty$, we have \begin{align}\label{eqn:cesaro-sum-limit}
   \sup_{t\leq T} \left|\frac{1}{n}\left[\sum_{i=1}^{\lfloor nt\rfloor }\pi_k(i)\right]\right| \leq \frac{1}{n}\left[\sum_{i=1}^{\lfloor nT\rfloor}|\pi_k(i)|\right] \to 0
\end{align} Therefore, the last term in~\eqref{eqn:integral-mart1} and~\eqref{eqn:integral-martk} converge to zero uniformly on $[0, T]$ for all $T>0$.

Now, observe that $X_k^n(t)$ and $\Phi^n(t)$ are Lipschitz functions (with Lipschitz constant being 2) with $X_k^n(0) = \frac{\ind\set{k=1}}{n}$ and $\Phi^n(0) = \frac{1}{n}$. Hence, the processes $X_k^n(t)$ and $\Phi^n(t)$ are tight in $D([0,T] \to \bR_+)$ for all $T>0$. Therefore for any subsequence $\set{n_m}$, there is a subsequence $\set{n_{m_l}}$ such that $X_k^{n_{m_l}}(t)$ and $ \Psi^{n_{m_l}}(t)$ converge. In~\cite{sethuraman2019sublinear}, it was shown that if $X_k^{n_{m_l}}(t) \to \varphi_k(t)$, then $\Psi^{n_{m_l}}(t) \to \sum_{k\geq 1} f(k)\varphi_k(t)$ when $\frac{f(k)}{k} \to 0$ (called \emph{(SUB)} assumption in~\cite{sethuraman2019sublinear}) and the convergence is uniform on every compact set of $[0,\infty)$.

From~\eqref{eqn:integral-mart1} and~\eqref{eqn:integral-martk}, the limiting random variables $\varphi_k(t)$ satisfy the following differential equations
\begin{align*}
    \dot\varphi_1(t) &= 1 - f(1) \frac{\varphi_1(t)}{\sum_{l\geq 1} f(l)\varphi_l(t)},\\
     \dot\varphi_k(t) &= f(k-1) \frac{\varphi_{k-1}(t)}{\sum_{l\geq 1} f(l)\varphi_l(t)} - f(k) \frac{\varphi_k(t)}{\sum_{l\geq 1} f(l)\varphi_l(t)}\qquad \text{ for } k\ge 2
\end{align*} 
with $\varphi_k(0) = 0$ for all $k$ (called `small' initial configuration in~\cites{sethuraman2019sublinear}).
It was shown in~\cite{sethuraman2019sublinear} that the above differential equations have a unique solution, and the solution is $\varphi_k(t) = p_k(f)t$. Letting $t=1$, we get 
\[
    \frac{N_k(n)}{n} \probc p_k(f) \qquad \text{for } k\geq 1.
\]
This proves Corollary~\ref{cor:meso-deg} in the sublinear regime. \\

\noindent \textbf{Linear regime:} Let $f(k) = k + \alpha$, then $\Psi(n) = (2+\alpha)n$. By Lemma~\ref{lem: deg-approx}, we have 
\begin{align*}
    \E(Z_1(n+1)) &= \left(1-\frac{1+\alpha}{(2+\alpha)n}\right)\E(Z_1(n)) + 1 + \E(\pi_1(n))
\end{align*}
and for any $k \geq 2$
\begin{align*}
    \E(Z_k(n+1)) &= \left(1-\frac{k+\alpha}{(2+\alpha)n}\right)\E(Z_k(n)) + \left(\frac{k-1+\alpha}{2+\alpha}\right)\frac{\E(Z_{k-1}(n))}{n}+ \E(\pi_k(n))
\end{align*} By Assumption~\ref{ass:lipschitz} and Lemma~\ref{lem: deg-approx}, we have $\lim\limits_{n\to\infty}\E(\pi_k) \to 0$ for all $k\ge 1$. Using~\cite{durrett-rg-book}*{Lemma 4.1.2}, we get that
\begin{align*}
    \lim_{n \to \infty} \E\left(\frac{Z_1(n)}{n}\right) &= \frac{2+\ga}{3+2\ga} = p_1(f), \\
    \lim_{n \to \infty} \E\left(\frac{Z_k(n)}{n}\right) &= \frac{k-1+\ga}{2+k+2\ga}\cdot p_{k-1}(f) = p_k(f),\quad k\ge 2.
\end{align*} 
Almost sure convergence follows from the bounded difference inequality (see~\cite{van2009random}*{Theorem 8.3}). This finishes the proof. \qed
\medskip

We end this section with an implication of the above proof.
\begin{lem}
    \label{eqn:normalizer}
    Under the assumptions of Cor.~\ref{cor:meso-deg}, the normalizing factor satisfies $\Psi(n)/n = \sum_{k=1}^\infty f(k) Z_k(n)/n \convas \lambda^* $ where $\lambda^* = \lambda^*(f)$ is the Malthusian rate of growth defined in~\eqref{eqn:malthus-def}.
\end{lem}
\begin{proof}
    In the linear regime $\Psi(n) = \sum_{v\in \cT(n)} (\deg(v,n) + \alpha) = 2(n-1) +n\alpha$ so that $\Psi(n)/n \convas 2+\alpha = \lambda^*$. For  the sublinear regime, this directly follows from step $2$ in the proof of~\cite{sethuraman2019sublinear}*{Theorem 2.3} which implies that
    \begin{align*}
    \frac{\Psi(n)}{n} \convas \sum\limits_{k=1}^\infty f(k)p_k(f) = \sum_{k=0}^\infty f(k)\frac{\gl^*}{\gl^* + f(k)}\prod_{i=0}^{k-1} \frac{f(i)}{\gl^* + f(i)} = \gl^* \sum_{k=0}^\infty \prod_{i=0}^{k} \frac{f(i)}{\gl^* + f(i)}.
\end{align*}
By definition in~\eqref{eqn:malthus-def}, $\gl^*$ is the unique positive constant satisfying, 
$\sum_{k=0}^\infty \prod_{i=0}^{k} \frac{f(i)}{\gl^* + f(i)} = 1$.  
\end{proof}

\subsection{Proof of Local weak convergence, Theorem~\ref{thm:meso-local}:}
\label{sec:pf-thm-lwc-meso}

We start with some terminology following~\cite{rudas2007random}. Let $\overrightarrow{\bbT}$ be the space $\bbT$ but where for every element $\vt$:
	\begin{enumeratea}
		\item \textcolor{black}{The vertices of $\vt$ are labelled using the Ulam-Harris set $\cN = \cup_{n=0}^\infty \bN^n$ with $\bN^0 = \set{\emptyset}$ denoting the root of $\vt$. This, in particular, gives the birth order of the children of each vertex.}
		\item Every edge has a direction from child to parent. For any vertex $v\in \vt$ other than the root, let $\fm(v)$ denote the parent of $v$ ($\fm$, a mnemonic for ``mother''). 
	\end{enumeratea}
  Note that the topology of any tree $\vt \in \overrightarrow{\bbT}$ is uniquely determined by its list of vertex labels. For $\vt\in \overrightarrow{\bbT}$, let $\cH(\vt)$ denote the collection of all historical orderings of $\vt$. More precisely,  $\sch_{\vt} = (v_0 = \emptyset, v_1,\ldots, v_{|\vt|-1}) \subseteq \cN^{|\vt|} $ is called a historical ordering of $\vt$ if it gives a possible birth ordering of vertices in $\vt$ starting from $\emptyset$ till its completion.  Formally for each $0\leq i\leq |\vt| -1$, $\cT(\sch_{\vt},i) := \set{v_0, v_1, \ldots, v_i} \in \overrightarrow{\bbT}$. 
  
 Given a fixed tree $\cT \in \overrightarrow{\bbT}$, for $0\leq i\leq |\cT|-1 $, let $\deg({v_i},\cT)$ denote the number of children of ${v_i}$ in $\cT$. Define the weight of $\cT$ as:
\[\vW(\cT) := \sum_{i=0}^{|\cT| -1}  f(\deg(v_i, \cT)),\]
where as before, $f$ is the attachment function. 
 Next given  $\vt \in \overrightarrow{\bbT}$ and a historical ordering $\sch_{\vt} \in \cH(\vt)$, define the following sequence of weight functionals: 
\textcolor{black}{\begin{align}
	\vW(\vt, \sch_{\vt}, k)&:= \vW(\cT(\sch_{\vt}, k)), \qquad 0\leq k\leq |\vt|-1, \notag \\ 
	 \vw_{\cic}(\vt, \sch_{\vt}, k+1) &:= f(\deg(\fm(v_{k+1}),\cT(\sch_{\vt}, k))), \qquad 0\leq k\leq |\vt|-2, \label{eqn:weight-order-def}
\end{align}}namely the weight associated to the number of children the parent of the vertex entering the system at step $k+1$, before the attachment at step $k+1$; as described in~\cite{rudas2007random}, conceptually this is the rate the associated continuous time branching process jumps from  $\cT(\sch_{\vt}, k)$ to $\cT(\sch_{\vt}, k+1)$. Note that for any $\vt$, the final term in the sequence is the same, irrespective of the ordering, 
\begin{align}
\label{eqn:314}
	\vW(\vt, \sch_{\vt}, |\vt|-1) = \sum_{i=0}^{|\vt|-1}  f(\deg(v_i, \vt)) = \vW(\vt). 
\end{align}
Next, given $\vt \in \overrightarrow{\bbT}$, note that each vertex $v\in \vt$ (with non-zero children) has an ordering of the children from oldest to youngest, and thus one can talk about the youngest child $y_{v}$ of $v$. For $0\leq v\leq |\vt|-1$ define the indicators,
\begin{align}
\label{eqn:ind-def}
	I_v:= \ind\set{v \mbox{ has child in $\vt$ and $y_{v}$ is a leaf in $\vt$}}. 
\end{align}
If $I_v= 1$, denote by $\vt^{\sss(v)}$ the tree obtained by deleting the youngest child of $v$ (which is necessarily a leaf). Now recall the Malthusian rate of growth parameter $\gl^*$ as in~\eqref{eqn:malthus-def}. Define the measure, 
\begin{align}
\label{eqn:pm-def}
	\fp(\vt):= \sum_{\sch_{\vt} \in \cH(\vt)} \frac{\gl^*}{\gl^*+\vW(\vt)} \prod_{k=0}^{|\vt|-2} \left[\frac{\vw_{\cic}(\vt, \sch_{\vt}, k+1)}{\gl^*+ \vW(\vt, \sch_{\vt}, k)}\right], \qquad \vt \in \overrightarrow{\bbT}. 
\end{align}
Here, the product above is taken to be one if $|\vt|=1$. Recall the asserted limit in Theorem~\ref{thm:meso-local} with fringe distribution $\varpi_{\BP_f}$ constructed as in Definition~\ref{def:limit-bp-meso}. 

\begin{prop}
	\label{prop:prop-of-fp}
	\begin{enumeratea}
		\item The measure $\fp(\cdot)$ in~\eqref{eqn:pm-def} is in fact a probability measure. Also, $\fp \equiv \fpm_{\BP_f}$ namely the fringe distribution associated to the $\BP_f$ as in Definition~\ref{def:limit-bp-meso}.  
		\item $\fp(\cdot)$ is the unique measure on $\overrightarrow{\bbT}$ satisfying the recursive equation 
		\begin{align}
		\label{eqn:recur-first-part}
			\fpm(\vt) = \frac{\sum_{v\in \vt} I_v  \cdot  \fpm(\vt^{\sss(v)}) f(\deg(v,\vt)-1)}{\gl^*+ \vW(\vt)},
		\end{align}
		with boundary conditions for $\vt = \set{v_0}$, i.e., a tree consisting of a single vertex given by 
		\begin{align}
		\label{eqn:recur-boundary}
			\fpm(\vt) = \frac{\gl^*}{\gl^*+ f(1)}. 
		\end{align}
	\end{enumeratea}
\end{prop}

\begin{proof}
Part(a) is proved in~\cite{rudas2007random}*{Theorem 2(b)}. To prove part (b) of the Proposition, it is enough to show the following two properties (as these can be recursively used to uniquely define $\fp$ via the formula in~\eqref{eqn:pm-def}):
\begin{enumeratei}
		\item $\fp$ satisfies the boundary conditions namely~\eqref{eqn:recur-boundary}. This is immediate by the definition of $\fp$ in~\eqref{eqn:pm-def}. 
		\item $\fp$ satisfies the recursive equation~\eqref{eqn:recur-first-part}: Now consider $\vt$ with $|\vt|\geq 2$. For ease of notation, for any vertex $u\in \vt$, write $d_u = \deg(u,\vt) $ for the degree in $\vt$.  Let $\cY_{\vt}:=\set{v: I_v =1}$ denote the collection of vertices whose youngest child is a leaf. Note that the historical orderings $\cH(\vt)$ can be partitioned as the disjoint union 
		\[\cH(\vt) = \sqcup_{u\in \cY_{\vt}} \set{\sch_{\vt} = (v_0, v_1, \ldots, v_{|\vt|-1}) \in \cH(\vt): v_{|\vt|-1} = y_u }=: \sqcup_{u\in \cY_{\vt}} \tilde{\cH}^{\sss(u)}(\vt),\] 
		, based on the identity of the final vertex (necessarily a leaf) that was adjoined to form $\vt$. \textcolor{black}{Further $\tilde{\cH}^{\sss(u)}(\vt)$ can be viewed as all historical orderings of the tree $\vt^{\sss(u)}$ followed by the addition of $y_u$ in the final step to obtain $\vt$. Thus, with a minor abuse of notation, we will write $\tilde{\cH}^{\sss(u)}(\vt) = {\cH}(\vt^{\sss(u)})$}. Now starting with the form of $\fp$ in~\eqref{eqn:pm-def}, using the above partition of $\cH(\vt)$ and the form of the weight sequences in~\eqref{eqn:weight-order-def}, we get, 
		\begin{align*}
			\fp(\vt) &= \sum_{u\in \cY_{\vt}} \sum_{\sch \in \cH(\vt^{\sss(u)})} \frac{\gl^*}{\gl^*+\vW(\vt)} \prod_{k=0}^{|\vt^{\sss(u)}|-2} \left[\frac{\vw_{\cic}(\vt^{\sss(u)}, \sch_{\vt^{\sss(u)}}, k+1)}{\gl^*+ \vW(\vt^{\sss(u)}, \sch_{\vt^{\sss(u)}}, k)}\right]\frac{f((d_u-1))}{(\gl^*+\vW(\vt^{\sss(u)}))}, \\
			&= \sum_{u\in \cY_{\vt}} \sum_{\sch \in \cH(\vt^{\sss(u)})} \frac{\gl^*}{\gl^*+\vW(\vt^{\sss(u)})} \prod_{k=0}^{|\vt^{\sss(u)}|-2} \left[\frac{\vw_{\cic}(\vt^{\sss(u)}, \sch_{\vt^{\sss(u)}}, k+1)}{2+ \vW(\vt^{\sss(u)}, \sch_{\vt^{\sss(u)}}, k)}\right]\frac{f((d_u-1))}{(\gl^*+\vW(\vt))}, \\ 
			&=  \frac{\sum_{u\in \cY_{\vt}} \fp(\vt^{\sss(u)})(f(d_u-1))}{\gl^*+\vW(t)}, \quad \mbox{using~\eqref{eqn:pm-def}.}
		\end{align*}
	\end{enumeratei}  
		This completes the proof. 
\end{proof}

\noindent {\bf Completing the Proof of Theorem~\ref{thm:meso-local}:} Recall the normalizing constant $\Psi(n) = \Psi_n$ in Lemma~\ref{eqn:normalizer}. For the rest of the proof, let $\Phi_n = n/\Psi_n$ so that by Lemma~\ref{eqn:normalizer}, $\Phi_n \convas (\gl^*)^{-1}$. From Section~\ref{sec:fringe-convg-def}, recall that for $\cT\in \overrightarrow{\bbT}$ and $v\in \cT$, $f_0(v,\cT) \in \overrightarrow{\bbT}$ denotes the fringe of vertex $v \in \cT$. For fixed $n\geq 1$, define the empirical count and fringe density in $\cT_n$ respectively via, 
\[
c_n(\vt):= \sum_{v\in \cT_n}\ind\set{f_0(v,\cT_n) = \vt}, \qquad 
\fP_n(\vt):= \frac{c_n(\vt)}{n}  \quad \text{for } \vt\in \overrightarrow{\bbT}.
\]  
The main goal for the rest of this section is to prove the following result. 
\begin{prop}
    \label{prop:exp-fring}
    Under the assumptions of Theorem~\ref{thm:meso-local}, with $\varpi_{\BP_f}$ as in Prop.~\ref{prop:prop-of-fp}, we have convergence in the expected fringe sense $\cT(n) \Efr \varpi_{\BP_f} $ as $n\to\infty$. 
\end{prop}
 \noindent {\bf Proof of Theorem~\ref{thm:meso-local} assuming Prop.~\ref{prop:exp-fring}:} As proved in~\cite{aldous-fringe} {\cg(see paragraph after \cite[Theorem 13]{aldous-fringe})}, $\varpi_{\BP_f}$ is an extremal fringe distribution. 
Theorem~\ref{thm:aldous-efr-pfr} now completes the proof. 

\qed

\noindent {\bf Proof of Proposition~\ref{prop:exp-fring}:}
We start with the following stochastic approximation Lemma. 
\begin{lemma}
\label{lem:approx}
    For any $\vt \in \overrightarrow{\bbT} $, there exists a sequence of random variables $\set{\epsilon_i(n,\vt)}_{n\geq 1}$ for $i=1,2$ such that
    \begin{align}
        \pr(c_{n+1}(\vt) &= c_n(\vt)-1\mid \cF_n) = \frac{c_n(\vt)}{n} \vW(\vt) \Phi_n + \epsilon_1(n,\vt)\label{eqn:tree-minus}\\
        \pr(c_{n+1}(\vt) &= c_n(\vt)+1\mid \cF_n) = \frac{1}{n} \sum_{u\in \cY_{\vt}} c_n(\vt^{\sss(u)}) f(d_u-1) \Phi_n + \epsilon_2(n,\vt)\label{eqn:tree-plus}
    \end{align}with $\epsilon_i(n,\vt) \to 0$ for $i=1,2$ almost surely and in $\cL^1$.
\end{lemma}
\begin{proof}
We first prove~\eqref{eqn:tree-minus}, when $f$ satisfies the assumptions \ref{ass:lipschitz} and $f$ is Lipschitz with Lipschitz constant $L \geq 0$.
    Let $V_n = \{i \in \cT_n : i \in f_0(v,\cT_n) = \vt \text { for some } v \in \cT_n \}$. Thus, we have
    \begin{align*}
        \pr(c_{n+1}(\vt) = c_n(\vt)-1\mid \cF_n) 
        &= \sum_{i=1}^n \ind\set{i \in V_n} \pr(\text{vertex }n+1 \text{ attaches to vertex } i \mid\cF_n).
    \end{align*}
\noindent Hence, by~\eqref{eqn:prob-decomp}, we have \begin{align*}
    \pr(c_{n+1}(\vt) = c_n(\vt)-1\mid \cF_n) &=  \sum_{i=1}^n \ind\set{i \in V_n}\frac{f(\deg(i,n))}{\Psi(n)} + \epsilon_1(n) = \frac{c_n(\vt)}{n}\vW(t)\Phi_n + \epsilon_1(n,\vt)
\end{align*}where $\epsilon_1(n,\vt) = -\left[\sum_{j=1}^4 R_j(n,\vt)\right]$ and $R_j(n,\vt) = \sum_{i=1}^n 
\ind\set{i\in V_n} w_j(i,n)$ for $1\leq j\leq 4$.

\noindent Let $D_{\vt}$ be the maximal degree of a vertex in $\vt$ and $M_\vt = \sup\limits_{j \leq D_\vt} f(j)$. We have 
      \begin{align*}
        |R_1(n,\vt)| &= \sum_{i =1}^n\ind\set{i \in V_n} |w_1(i,n)| = \sum_{i=1}^n\ind\set{i \in V_n}\frac{f(\deg(i,n))}{\Psi(n)}\pr(n-n^{\gb}\xi < i)\\
        &\leq \frac{M_\vt}{f_*n} \left(\sum_{i=1}^n\pr(n-n^{\gb}\xi < i)\right) \leq \frac{M_\vt}{f_*} \E\left(n^{\beta-1}\xi \cdot \ind\set{n-n^{\gb}\xi \geq 1}\right).
      \end{align*}
By DCT, we have $R_1(n,\vt)\to 0$. Also,
    \begin{align*}
    |R_2(n,\vt)| \leq \sum_{i =1}^n\ind\set{i \in V_n} |w_2(i,n)|
       &\leq \frac{2LM_\vt}{f_*^2}\E\left(\frac{n^{2\gb-1}\xi^2\cdot\ind\set{n-n^{\gb}\xi \geq 1}}{\lfloor n-n^{\gb}\xi \rfloor } \right)\\
        &\leq \frac{2LM_\vt}{f_*^2}\E\left(\frac{n^{\beta}\xi \cdot \ind\set{n-n^{\gb}\xi \geq 1}}{\lfloor n-n^{\gb}\xi \rfloor } \right)
    \end{align*}
and 
    \begin{align*}
        |R_3(n,\vt)| =  \sum_{i =1}^n\ind\set{i \in V_n} w_3(i,n)
        &\leq \frac{M_\vt}{2f_*}\E\left(\frac{n^{\alpha}\xi\cdot \ \ind\set{n-n^{\gb}\xi \geq 1}}{\lfloor n-n^{\gb}\xi \rfloor } \right).
    \end{align*}

Therefore, under the assumptions~\ref{ass:lipschitz}, we have $\lim\limits_{n \to \infty}R_j(n,\vt) \to 0$ for $j=2,3$. As observed earlier, we have $|R_4(n,\vt)| \leq \pr(\xi \geq n^{1-\beta}) \to 0$. Hence, $\epsilon_1(n,\vt) \to 0$ almost surely and $\cL^1$.

\begin{rem}
Note that in the argument above, the Lipschitz property of $f$ is used only to bound the term $R_2(n,\vt)$. Following the argument used in the proof of \ref{lem: deg-approx}, when $f$ is sublinear and non-decreasing, as $n\to \infty$, we have$$\left|R_2(n,\vt) \right| \leq \frac{M_{\vt}}{f_*} \E\left(\frac{n^\beta \xi \cdot \ind\set{n-n^\beta\xi \geq 1}}{\lfloor n-n^\beta\xi \rfloor}\right) \to  0.$$
\end{rem}

\noindent To prove~\eqref{eqn:tree-plus}, define $U_n = \set{v \in \cT_n: \exists u \text{ such that }v \in f(u,\cT_n) = \vt \text{ and }  v\in \cY_{f(u,\cT_n)} }$. We have,
  \begin{align*}
    \pr(c_{n+1}(\vt) = c_n(\vt)-1\mid \cF_n) &= \sum_{v \in \cT_n}\ind\set{v \in U_n} \pr(n+1 \to v|\cF_n).
  \end{align*}
   Again, using~\eqref{eqn:prob-decomp} and similar arguments as above, we get~\eqref{eqn:tree-plus}.
\end{proof}
  For the rest of the argument, $\vt$ will be a fixed element of $\overrightarrow{\bbT}$ and let $\fN_{n}(\vt) = \E(c_n(\vt))$. Using the dynamics of the process $\set{\cT_n:n\geq 2}$, the resultant dynamics for the above counts (for fixed $\vt$) are via Lemma~\ref{lem:approx}, we have \begin{align}\label{eqn:evol-c_n}
	\E(c_{n+1}(\vt)\mid \cF_n) - c_n(\vt) = \frac{1}{n}\bigg[\sum_{u\in \vt} \Phi_n c_n(\vt^{\sss(u)}) f(d_u -1)I_u  - \vW(\vt)\Phi_n c_n(\vt)\bigg] + \epsilon(n)
\end{align}where $\epsilon_n(\vt) = \frac{1}{n+1} \left[\epsilon_1(n,\vt) - \epsilon_2(n,\vt)\right]$ which converges to zero by Lemma~\ref{lem:approx}. Taking expectations we get \begin{align}\label{eqn:exp-recursion-tree}
    \fN_{n+1}(\vt) = -\frac{\vW(\vt)}{\lambda^*} \frac{\fN_n(\vt)}{n} +  \frac{1}{\lambda^*}\left(\sum_{u\in \vt} \frac{\fN_n(\vt^{\sss(u)})}{n}f(d_u -1)I_u\right) + \hat{\epsilon}_n(\vt)
\end{align}where \begin{align*}
    \hat{\epsilon}_n(\vt) = \bigg[\sum_{u\in \vt} \E\left(\left[\Phi_n-\frac{1}{\lambda^*}\right] \fP_n(\vt^{\sss(u)})\right) f(d_u -1)I_u  - \vW(\vt) \E\left( \left[\Phi_n-\frac{1}{\lambda^*}\right]\fP_n(\vt)\right)\bigg] + \epsilon_n(\vt)
\end{align*} Under the assumptions~\ref{ass:attach-func}, we have $f_* = \inf_{k\geq 1 }f(k) > 0$ and $\Psi(n) = \sum\limits_{k=1}^\infty Z_k(n)f(k) \geq nf_* $. Hence we have $\Phi_n \leq (f_*)^{-1}$ for all $n$. Using DCT, we have $\hat{\epsilon}_n(\vt) \to 0$ almost surely and in $\cL^1$.

Note that in~\eqref{eqn:exp-recursion-tree}, the second term on the right-hand side involves the expectation of empirical proportions of trees, which are strict sub-trees of $\vt$. This suggests that if expected proportions of these trees converge, then under regularity conditions, so should $\E(\fP_n(\vt))$, where the limit satisfies a recursive equation inherited from~\eqref{eqn:exp-recursion-tree}. 

Now by Corollary~\ref{cor:meso-deg},  if $\vt$ is a tree consisting of a single vertex $\vt = \set{v}$ (namely, the corresponding fringe density is just the leaf density), then 
\begin{align}
\label{eqn:141}
	\E(\fP_n(\vt)) = \frac{\fN_n(\vt)}{n} \to \varpi(\vt):= \frac{\gl^*}{\gl^*+f(1)} = \fp(\vt).
\end{align}
Fix $l\geq 2$ and assume that for 
\begin{align}
\label{eqn:hypo-vt}
	\forall~\tilde{\vt} \in \overrightarrow{\bbT}  \mbox{  with } |\tilde{\vt}| \leq (l-1), \qquad \E(\fP_n(\tilde{\vt})) = \frac{\fN_n(\tilde{\vt})}{n} \to \varpi(\tilde{\vt}), \qquad \mbox{ as } n\to \infty, 
\end{align}
for some positive limit constant $\varpi(\tilde{\vt})$ which satisfies the boundary condition in~\eqref{eqn:141}. Let $\vt \in\overrightarrow{\bbT} $ with $|\vt| = l$ and consider the evolution equation in~\eqref{eqn:exp-recursion-tree}. By~\cite{durrett-rg-book}*{Lemma 4.1.2}, we have \begin{align*}
    \lim\limits_{n\to \infty}\frac{\fN_n(\vt)}{n} = \frac{\left(\sum\limits_{u \in \vt}\fpm(\vt^{\sss(u)})f(d_u -1)I_u\right)}{\lambda^*+\vW(\vt) } 
\end{align*} satisfies the recursion in~\eqref{eqn:recur-first-part}. Therefore by Proposition~\ref{prop:prop-of-fp}, $\E(\fP_n(\vt)) \to \fp(\vt)$ for all $\vt \in \overrightarrow{\bT}$. This completes the proof of expected fringe convergence, and thus Proposition~\ref{prop:exp-fring}.
\qed

\subsection{Second order fluctuations in the linear regime, Proof of Theorem~\ref{thm:meso-degree-CLT}:}
\label{sec:second-order}
We largely follow~\cite{resnick2015asymptotic} where the proof involves a careful evaluation of the dynamics of the number of leaves, coupled with the Lindeberg-Feller Martingale CLT stated in~\cites{resnick2015asymptotic}. We sketch the main changes in the proof of~\cite{resnick2015asymptotic}.   
\begin{lemma}[Martingale CLT]
\label{lem:mg-clt}
Let $\set{X_{n,m}, \cF_{n,m}}_{1\leq m\leq n}$ be a square-integrable martingale array such that \begin{enumeratei}
    \item $\lim\limits_{n\to \infty}\sum\limits_{m=1}^n \E(X_{n,m}^2\mid \cF_{n,m-1}) \probc \sigma^2$. 
    \item $\lim\limits_{n\to \infty}\sum\limits_{m=1}^n \E(X_{n,m}^2\ind\set{|X_{n,m}| \geq \epsilon}\mid \cF_{n,m-1}) \probc 0$ for all $\epsilon > 0$.
\end{enumeratei} 
Then 
$
    \sum_{m=1}^n X_{n,m} \convd \N(0,\sigma^2)
$ as $n\to \infty$. 
\end{lemma}

We now provide a proof of Theorem \ref{thm:meso-degree-CLT}. 
\begin{proof}[Proof of Theorem \ref{thm:meso-degree-CLT}]

Now recall $Z_1(n) := Z_n$ the number of degree $1$ vertices in $\cT(n)$. We consider the case of affine attachment function $f(k) = k+\alpha, \quad k\geq 1$. From Cor.~\ref{cor:meso-deg}, we have 
\[Z_n \convas \frac{2+\alpha}{3+2\alpha}:= p_1,\] 
when~\eqref{eqn:815} holds. Let $\gamma := {(1+\alpha)}/{(2+\alpha)}$. For any $n \geq 1$, we have \begin{align*}
    \E(Z_{n+1}\mid \cF_n) &=  Z_n + 1- \frac{1+\alpha}{2+\alpha} \sum_{i=1}^n \ind\set{\deg(i,n) =1} \E\left(\frac{\ind\set{n-n^\beta\xi \geq i}}{\lfloor n-n^\beta\xi\rfloor}\right)\\
    &\hspace{2.5cm}- \ind\set{\deg(1,n)=1}\pr(n-n^\beta \xi < 1)\\
    &= Z_n + 1- \frac{1+\alpha}{2+\alpha} \sum_{i=1}^n \ind\set{\deg(i,n) =1} \E\left(\frac{\ind\set{n-n^\beta\xi \geq 1}}{\lfloor n-n^\beta\xi\rfloor}\right) \\
    &\hspace{2cm}+ \frac{1+\alpha}{2+\alpha} \sum_{i=1}^n \ind\set{\deg(i,n) =1} \E\left(\frac{ \ind\set{1 \leq n-n^\beta\xi < i}}{\lfloor n-n^\beta\xi\rfloor}\right)\\
    &\hspace{2.5cm}- \ind\set{\deg(1,n)=1}\pr(n-n^\beta \xi < 1)\\
    &= k_n Z_n + A_n,
\end{align*}
where \begin{align*}
    k_n &= 1-\gamma\E\left(\frac{\ind\set{n-n^\beta\xi \geq 1}}{\lfloor n-n^\beta\xi\rfloor}\right) \text{ and}\\ 
    A_n &= 1 + \gamma \ \sum_{i=1}^n \ind\set{\deg(i,n) = 1} \E\left(\frac{\ind\set{1\leq n-n^\beta \xi <i }}{\lfloor n-n^\beta\xi\rfloor}\right)\\
    &\qquad\qquad\qquad\qquad- \ind\set{\deg(1,n)=1}\pr(n-n^\beta \xi < 1).
\end{align*} 
We then have $M_n$ is a $\cF_n$ adapted martingale where \begin{align*}
    M_n = \prod_{j=1}^{n-1}k_j^{-1}\cdot Z_n - \sum\limits_{i=1}^{n-1} \left( \prod_{j=1}^{i-1}k_j^{-1}\cdot A_i\right)
\end{align*}
Let $d_{n} = M_n - M_{n-1}$ be the martingale difference. We have \begin{align*}
    d^2_{n+1} = \prod\limits_{j=1}^n k_j^{-2}\cdot \biggl[(Z_{n+1} - Z_n) + 2(Z_{n+1}-Z_n)\left[(1-k_n)Z_n - k_n A_n\right] + \left[\right(1-k_n)Z_n - k_nA_n]^2\biggr].
\end{align*}
Hence, \begin{align*}
    \E(d^2_{n+1}\mid\cF_n) &= \prod\limits_{j=1}^n k_j^{-2} \cdot \biggl[\beta_1(n) + \beta_2(n)\biggr],
\end{align*}
where, 
    $\beta_1(n) = \E(Z_{n+1}-Z_n\mid\cF_n) \left(1+2(1-k_n)Z_n -2k_nA_n\right)$ and 
    $\beta_2(n) = \left(Z_n(1-k_n) - k_nA_n\right)^2.$
Under the assumptions~\ref{ass:lipschitz} and~\eqref{eqn:1257}, we have \begin{align*}
     \lim\limits_{n \to \infty}|A_n-1| &\leq \lim\limits_{n\to \infty}\left[ \gamma \E\left(\frac{n^\gb\xi
     \cdot\ind\set{n-n^\gb\xi \geq 1}}{\lfloor n-n^\beta\xi\rfloor}\right) + \pr(n-n^\beta\xi < 1)\right] = 0\\   
    &\lim\limits_{n\to \infty} n(1-k_n) =\gamma,
    \quad \lim\limits_{n \to \infty} n^{\gamma}\prod\limits_{j=1}^n k_j = C,
\end{align*}
for some $C > 0$ is a constant. Hence, we have $\E(d^2_{n+1}\mid\cF_n) \sim \frac{(1+\alpha)(2+\alpha)}{C^2(3+2\alpha)^2} n^{2\gamma}$ and  \begin{align}\label{eqn:cond1}
    \sum\limits_{i=1}^n \E(d_{i+1}^2\mid \cF_n) \sim  \frac{(1+\alpha)(2+\alpha)}{C^2 (3+2\alpha)^2 (2\gamma+1)} n^{2\gamma+1} = \sigma^2 n^{2\gamma+1} 
\end{align} 
where the asymptotic equivalence holds almost surely. With the above decompositions, using Lemma~\ref{lem:mg-clt} and following the same argument as~\cite{resnick2015asymptotic} now shows that

\begin{align}\label{eqn:clt-mg-term}
    \frac{M_n}{n^{\gamma+1/2}} \convd \N(0,\sigma^2).
\end{align}
We have\begin{align*}
    \frac{M_n}{n^{\gamma+1/2}} &= \frac{\sqrt{n}}{n^\gamma  \prod_{j=1}^{n-1}k_j} \left(\frac{Z_n}{n} - \frac{\fX_n}{n}\right)
\end{align*}
where $\fX_n = \sum\limits_{i=1}^{n-1} \left(\prod\limits_{j=i}^{n-1} k_j\right)A_j$ is the random centering in Theorem~\ref{thm:meso-degree-CLT}. Note that $\fX_n$ satisfies the recursion \begin{align}\label{eqn:recursion-X}
    \fX_{n+1} &= k_n \fX_n + A_n = \left(1-\gamma \E\left(\frac{\ind\set{n-n^\beta \xi \geq 1}}{\lfloor n-n^\beta \xi \rfloor}\right)\right) \fX_n + A_n.
\end{align} 
Since, $|A_n-1|\leq \gamma \E\left(\frac{n^\beta \xi\cdot \ind\set{n-n^\beta\xi \geq 1}}{\lfloor n-n^\beta\xi \rfloor}\right) \to 0$ under Assumption~\ref{ass:lipschitz}. Therefore, by~\cite{durrett-rg-book}*{Lemma 4.1.2} we have \begin{align*}
    \frac{\fX_n}{n} \convas \frac{1}{\gamma+1} = p_1.
\end{align*} 
Moreover, $n^\gamma \prod\limits_{j=1}^{n-1}k_j \to C$ implies that
\begin{align*}
    \sqrt{n}\left(\frac{Z_n - \fX_n}{n}\right) \convd \N(0,\sigma_1^2)
\end{align*}
where $\sigma_1^2 = \sigma^2 C^2 =\frac{(1+\alpha)(2+\alpha)^2}{(3+2\alpha)^2(4+3\alpha)}.$ This proves the first part of the Theorem. 

We now prove the final assertion of the theorem, namely showing under the sufficient condition mentioned in Theorem~\ref{thm:meso-degree-CLT}, $\epsilon_n/n\to 0$ where $\epsilon_n = \sqrt{n} (\fX_n
- np_1)$. Using~\eqref{eqn:recursion-X} and $p_1 = 1-\gamma p_1$, we have \begin{align*}
    \epsilon_{n+1} = \left[1-\gamma \E\left(\frac{\ind\set{n-n^\beta\xi \geq 1}}{\lfloor n-n^\beta \xi \rfloor}\right)\right]\epsilon_n +& \gamma \sqrt{n} \left[1-\E\left(\frac{n\cdot\ind\set{n-n^\beta \xi \geq 1}}{\lfloor n-n^\beta \xi \rfloor}\right)\right]p_1 +\sqrt{n} (A_n -1)
\end{align*} and \begin{align*}
   \left| \left[1-\E\left(\frac{n\cdot\ind\set{n-n^\beta \xi \geq 1}}{\lfloor n-n^\beta \xi \rfloor}\right)\right]p_1 + (A_n -1)\right| \leq (1+\gamma)&\E\left(\frac{n^\beta\xi \cdot \ind\set{n-n^\beta\xi \geq 1}}{\lfloor n-n^\beta \xi \rfloor}\right)\\
   &+ \ind\set{\deg(1,n)=1} \pr(n-n^\beta\xi < 1).
\end{align*} Since, $$\lim_{n\to \infty}\pr(\deg(1,n) = 1) = \prod\limits_{j=1}^\infty\left[1- \E\left(\frac{\ind\set{n-n^\beta\xi \geq 1}}{\lfloor n-n^\beta \xi \rfloor}\right)\right] = 0,$$ it follows that $\sqrt{n}\pr(n-n^\beta\xi)\ind\set{\deg(1,n) = 1} \convas 0$. Therefore, if $ \lim\limits_{n\to \infty}\sqrt{n} \E\left(\frac{n^\beta\xi\cdot\ind\set{n-n^\beta\xi \geq 1}}{\lfloor n-n^\beta \xi \rfloor}\right) \to 0$, using~\cite{durrett-rg-book}*{Lemma 4.1.2}, we have \begin{align*}
    \frac{\epsilon_n}{n} = \sqrt{n} \left(\frac{\fX_n}{n} -p_1\right) \probc 0
\end{align*} and hence by~\eqref{eqn:clt-mg-term}, we have \begin{align*}
    \frac{Z_n - np_1}{\sqrt{n}} \convd \N(0,\sigma_1^2).
\end{align*}This finishes the proof.     
\end{proof}

\subsection{Root degree asymptotics}
\label{sec:root-deg}
Here, we will prove Theorem~\ref{thm:meso-max-degree} on the scaling of the root degree $M(\rho,n)$ in the mesoscopic regime. In the following, $C, C', C_1, C_2,\ldots$ will denote generic positive finite constants whose values might change between lines.

We add a self-loop to the root vertex at time $1$ so that the sum of degrees at time $n$ is always $2n$. We break the root vertex into two vertices: $0'$ and $1'$. Vertex $0'$ has degree $0$ at time $1$, and vertex $1'$ has a self-loop with degree $2$ at time $1$. At any time $(n+1)$, vertex $0'$ will get an edge from vertex $(n+1)$ if $\fplus{n-n^{\gb}\xi_{n+1}}<1$; whereas when $\fplus{n-n^{\gb}\xi_{n+1}}\ge1$, vertex $0'$ will be connected with probability $\frac{\deg(0',\fplus{n-n^{\gb}\xi_{n+1}})}{(2+\ga)\fplus{n-n^{\gb}\xi_{n+1}}}$ and vertex $1'$ will be connected with probability $\frac{\deg(1',\fplus{n-n^{\gb}\xi_{n+1}})+\ga}{(2+\ga)\fplus{n-n^{\gb}\xi_{n+1}}}$. Note that, 
\[
M(\rho,n) = \deg(0',n)+\deg(1',n)
\]
for all $n\ge 1$. One can think of the edges to vertex $0'$ as all the edges created by a ``large'' delay and modified by the linear preferential attachment mechanism. We define $$M_n:=\deg(0',n),\ T_n:=\deg(1',n)+\ga\text{ and  }\theta=1/(2+\ga).$$ Note that both $M_n$ and $T_{n}$ are increasing with $n$. The heuristic is that $T_n$ grows at the order of $n^\theta$, irrespective of the delay when $\E[X\log_+X]<\infty$. However, the growth of $M_n$ is given by $n^\theta$ when $\E[X^{1-\theta}]<\infty$ and by $n^\theta \E[(X\wedge n)^{1-\theta}]$ otherwise. \\

\noindent (i)~Note that the transition from time $n$ to $(n+1)$ is  $M_n \leadsto M_n +1$ with conditional probability (on the natural filtration $\set{\cF_n:n\geq 1}$),
\begin{align}\label{smart}
     &\E_{\xi} \bigg( 
    \frac{\theta  M_{\fplus{n-n^\gb\xi}} }{ \fplus{n-n^\gb\xi}}
    \cdot \ind\set{n-n^\gb\xi\ge1} 
    \bigg) 
    + \pr(\xi>(n-1)/n^\gb)\notag\\
   & = \frac{\theta}{n}\cdot M_n\cdot  (1+\eps_n) - \theta  \E_{\xi} \bigg( 
    \frac{M_n-M_{\fplus{n-n^\gb\xi}} }{ \fplus{n-n^\gb\xi}}
    \cdot \ind\set{n^\gb\xi\le n-1} 
    \bigg) + \gamma_n,
\end{align} 
where
\[
\eps_n = \E_{\xi}\bigg( \frac{n - \fplus{n-n^\gb\xi} }{\fplus{n-n^\gb\xi}}\cdot \ind\set{n-n^{\gb}\xi \ge 1}\bigg), \text{ and } \gamma_n = \pr(\xi>n^{1-\gb}(1-1/n)) .
\]
 Thus, letting $\Theta_n = \prod_{j=1}^{n-1} \left(1+\theta(1+\eps_i)/i\right)$ (with $\Theta_1 \equiv 1$), it is easy to check that the sequence, 
\begin{align}\label{eq:supmart}
\fR_n = \frac{M_n}{\Theta_n} - \sum_{j=1}^{n-1} \frac{\gamma_j}{\Theta_{j+1}}, \qquad  n\geq 1,
\end{align}
is a supermartingale.  Similarly, the transition from time $n$ to $(n+1)$ is  $T_n \leadsto T_n +1$ with conditional probability (on the natural filtration $\set{\cF_n:n\geq 1}$),
\begin{align}\label{smart2}
     &\E_{\xi} \bigg( 
    \frac{\theta  T_{\fplus{n-n^\gb\xi}} }{ \fplus{n-n^\gb\xi}}
    \cdot \ind\set{n-n^\gb\xi\ge1} 
    \bigg)\notag\\
    &= \frac{\theta}{n}\cdot T_n\cdot  (1+\eps_n)- \theta  \E_{\xi} \bigg( 
    \frac{T_n-T_{\fplus{n-n^\gb\xi}} }{ \fplus{n-n^\gb\xi}}
    \cdot \ind\set{n^\gb\xi\le n-1} 
    \bigg),
\end{align}  
and we get that 
\begin{align}\label{eq:supmart2}
\fT_n = \frac{T_n}{\Theta_n}, \qquad  n\geq 1,
\end{align}
is a non-negative supermartingale and hence converges almost surely.

Under condition in~\eqref{eqn:1257}, it is easy to check that $\Theta_n \sim n^\theta$. Assuming~\eqref{eqn:815}, and with $X=\xi^{1/(1-\gb)}$, we have 
\begin{align}\label{eq:finsum}
\sum_{j=1}^{\infty} \frac{\gamma_j}{\Theta_{j+1}} \le \text{const.}\cdot \sum_{j=1}^{\infty} j^{-\theta} \pr(X\ge j)\le \text{const.}\cdot\E[X^{1-\theta}]<\infty,
\end{align}
and thus $\sup_{n\geq 1} \E(\fR_n^{-}) < \infty$. The supermartingale convergence theorem now gives almost sure convergence for $\fR_{n}$.  

To prove $L^2$ convergence, it is enough to show that $\{n^{-2\theta}\E M_n^2, n\ge 1\}$ and $\{n^{-2\theta}\E T_n^2, n\ge 1\}$ are uniformly bounded. We will prove the first; the other follows similarly. Observe that using~\eqref{smart}, we have
\begin{align*}
\E\left(M_{n+1}^2 \ \vert\ \cF_n\right) &\le M_n^2 + (2M_n + 1)\left[\frac{\theta(1+ \eps_n)}{n}M_n + \gamma_n\right]\\
&\le \left(1 + \frac{2\theta(1+ \eps_n)}{n}\right)M_n^2 + CM_n\left(\frac{1}{n} + \gamma_n\right), n\ge 1.
\end{align*}
Define 
\[
\Theta_1^{(2)} \equiv 1, \qquad \Theta_n^{(2)}:= \prod_{j=1}^{n-1} \left(1+\frac{2\theta(1+\eps_i)}{i}\right),\quad n>1.
\]
Clearly, $\Theta_n^{(2)} \sim \Theta_n^2\sim n^{2\theta}$ for $n$  large.
By the supermartingale analysis above, $\E M_n \le C \Theta_n$ for all $n \ge 1$.
Using induction, we have
\[
\frac{\E M_n^2}{\Theta_n^{(2)}} \le \E M_1^2 + C\cdot \sum_{i=1}^{n-1}\frac{\Theta_i\ }{\Theta_{i+1}^{(2)}}\left(\frac{1}{i} + \gamma_i\right)
\le C'\cdot \left(\sum_{i=1}^{\infty}\frac{1}{i^{1+\theta}} + \sum_{i=1}^{\infty}\frac{\gamma_i}{\Theta_{i+1}}\right)<\infty,
\]
by~\eqref{eq:finsum} and the fact that $\theta>0$. 

We have $\{n^{-\theta}(M_n+T_{n}), n \ge 1\}$ is uniformly integrable. 
To complete the proof of (i), it suffices to show that $\liminf_{n \to \infty} n^{-\theta}\E(T_n) >0$. 
From equation~\eqref{eq:supmart2}, we have
\begin{align}\label{eq:imp}
\E \frac{T_n}{\Theta_n} \le \E \frac{T_k}{\Theta_k},
\end{align}
for all $k\le n$. 
In particular, we have for all $n\ge 1$, with $a=2/(1-\gb)>2$,
\begin{align}
\begin{split}
    \E T_{n+1} 
&=\E T_{n}  + \theta  \E \bigg( 
    \frac{T_{\fplus{n-n^\gb\xi}} }{ \fplus{n-n^\gb\xi}}
    \cdot \ind\set{n-n^\gb\xi\ge1} 
    \bigg)\notag \\
    & \ge \E T_{n}  + \theta  \E \bigg( 
    \frac{T_{\fplus{n-n^\gb\xi}} }{ \fplus{n-n^\gb\xi}}
    \cdot \ind\set{n-n^\gb\xi\ge a} 
    \bigg)\notag \\
    & \ge \E T_{n} \left(1+\frac{\theta}{n}\cdot \E\left(\frac{\Theta_{\fplus{n-n^\gb\xi}}}{\Theta_{n}}\frac{n}{\fplus{n-n^\gb\xi}}\cdot \ind\set{n-n^\gb\xi\ge a} \right)\right)\notag\\
    & \ge \E T_{n} \left(1+\frac{\theta}{n}-\frac{\theta}{n}\cdot \E\left(\frac{\Theta_n-\Theta_{\fplus{n-n^\gb\xi}}}{\Theta_n}\cdot \ind\set{n-n^\gb\xi\ge a} + \ind\set{\xi>n^{1-\gb}\left(1-a/n\right)} \right)\right)
\end{split}
\end{align}
Thus, to show that $\E T_{n}/n^{\theta}$ is bounded away from $0$, it is enough to show that
\[
\sum_{n=1}^\infty \frac{1}{n}\cdot \E\left(\frac{\Theta_n-\Theta_{\fplus{n-n^\gb\xi}}}{\Theta_n}\cdot \ind\set{n-n^\gb\xi\ge a} \right) <\infty
\]
and
\[\sum_{n=1}^\infty\frac{1}{n}\pr\left(X>n\cdot \left(1-a/n\right)^{1/(1-\gb)}\right)<\infty.
\]
The second condition above is satisfied because of the assumption $\E\left[X^{1-\theta}\right] < \infty$. To verify the first condition,
note that $\Theta_{n}-\Theta_{n-k}\le C\cdot k(n-k)^{\theta-1}$ for all $0\le k<n$. Hence, we have
\begin{align*}
\sum_{n=1}^\infty \frac{1}{n} &\cdot \E\left(\frac{\Theta_n-\Theta_{\fplus{n-n^\gb\xi}}}{\Theta_n}\cdot \ind\set{n-n^\gb\xi\ge a} \right)\\
&\le C\E\left(\sum_{n=1}^\infty \frac{X^{1-\gb}}{n^{1+\theta-\gb}  \lfloor (1- \beta)(n-X)\rfloor^{1-\theta}}\cdot \ind\set{n - X\ge a} \right)\\
&\le  C\E\left(\sum_{n=1}^\infty \frac{X^{1-\gb\vee\theta}}{n^{1+\theta-\gb\vee\theta}  (n-X)^{1-\theta}}\cdot \ind\set{X\le n-1} \right) \\
&\le  C \E\left(\sum_{n\ge  X+1} \frac{X^{1-\gb\vee\theta}}{ (n-X)^{2-\gb\vee\theta}}\right)\le C\E(X^{1-\gb\vee\theta})<\infty.
\end{align*}

Here, we used the fact that, for $n> x$, we have
\begin{align*}
    n-x\ge n-n^\gb x^{1-\gb} \ge \int_{x^{1-\gb}}^{n^{1-\gb}} t^{\gb/(1-\gb)}\;dt = (1-\gb)(n-x)
\end{align*}in the first inequality and $n^{1+\theta-\gb\vee\theta}  (n-X)^{1-\theta} \ge (n-X)^{2-\beta\vee\theta}$ for $X \le n-1$ in the third inequality. This concludes the proof of (i).

(ii) Note that when $\E X=\infty$, we have a large number of extra connections to the root due to delay. \textcolor{black}{In particular, 
\[
M(\rho,n)\ge \sum_{i=1}^{n-1} \ind\set{X_i\ge\ i\,(1-1/i)^{{1/(1-\gb)}}} \ge S_n:= \sum_{i=1}^{n-1} \ind\set{X_i\ge\ i},
\]
for $n \ge 2$, where the collection $\set{X_i=\xi_{i+1}^{1/(1-\gb)}}_{i\ge 1}$} are i.i.d. Clearly, 
\begin{align*}
    \E S_n &= \sum_{i=1}^{n-1} \pr(X\ge\ i)
 = \E (\lfloor X\rfloor\wedge (n-1)) =\E(X\wedge n)+O(1)
\end{align*}
and \textcolor{black}{$\var(S_n)\le \E(S_n) =\E(X\wedge n)+O(1)$}. Using Chebyshev's inequality, for any $\eps>0$, we get $\pr\left(M(\rho,n)\ge (1-\eps)\E (X\wedge n)\right)\to 1$ as $n\to\infty$. This completes the proof.\qed

\section{Conclusion}
\label{sec:conc}
The main goal of this paper is to formulate models of network evolution that incorporate delay as well as preferential attachment and initiate the study of asymptotics of specific functionals. This class of models suggests a plethora of open directions, which we now describe. 
\begin{enumeratea}
\item {\bf Macroscopic regime:} This paper dealt with the setting $\beta <1$. The setting $\beta =1$, namely where the amount of ``information'' available to new individuals is the structure of the network a non-trivial density of time in the past, is studied in~\cite{BBDS04_macro} where it is shown that, at least when the attachment function is linear, the local weak limit of the sequence of networks is very different from the setting without delay, and where in many natural settings, the degree distribution admits a strictly heavier tail exponent from the regime without delay. 
\item {\bf Non-{\it i.i.d.}~delays:} The paper states results for {\it i.i.d.}~delays; however, if one looks at the proofs, many of the results will go through assuming a sequence of independent delays $\set{\xi_n:n\geq 1}$ that satisfy similar regularity properties, for example in the mesoscopic regime, conditions such as~\eqref{eqn:815} and~\eqref{eqn:1257} as $n\to\infty$. We leave such extensions (and limits thereof) to other interested researchers.  
    \item {\bf Non-tree regime:} To keep this paper to a manageable length, we have dealt with the setting where the network stream is a collection of trees. All of the main results for both the meso and macro regimes should be extendable to the more general network setting where vertices enter with more than one edge that they recursively use to connect via probabilistic choices with information limited by delay. The no delay setting has witnessed significant interest over the last few years~\cites{garavaglia2020local,garavaglia2022universality,banerjee2023local}. 
    \item {\bf General attachment functions:} This paper while admitting general attachment functions $f$, assumed specific smoothness properties on the function to facilitate the analysis of various functionals driving the stochastic evolution. Extending the results to more general functions, and in particular, finding settings where, for example, local asymptotics are \emph{different} from the regime without delay, would be quite exciting and worth pursuing. 
\end{enumeratea}

\section*{Acknowledgements}
Banerjee was partially supported by the NSF CAREER award DMS-2141621. Bhamidi and Sakanaveeti were partially supported by NSF DMS-2113662, DMS-2413928, and DMS-2434559. Banerjee and Bhamidi were partially funded by NSF RTG grant DMS-2134107. Part of
this material is based upon work supported by the National Science Foundation under Grant
No. DMS-1928930, while Banerjee and  Bhamidi were in residence at the Simons Laufer
Mathematical Sciences Institute in Berkeley, California, during the Spring 2025 semester. We thank Charles Cooper for writing the Python program to simulate the model and Prabhanka Deka for enlightening discussions and comments related to the paper. The simulations in Fig.~\ref{fig:sim} were plotted using the excellent Graph tools \cite{peixoto_graph-tool_2014}. {\cg We thank the referee for a detailed evaluation which significantly improved the paper. }

\bibliographystyle{chicago}\bibliography{pad.bib,pref_change_bib.bib,scaling.bib,persistence.bib,ref.bib}

\end{document}